\documentclass[letterpaper,10pt]{article}

\usepackage[margin=0.5in]{geometry}

\usepackage{fullpage}
\usepackage{enumerate}
\usepackage{authblk}
\usepackage{verbatim}
\usepackage{section, amsthm, textcase, setspace, amssymb, amscd, lineno, amsmath, amsfonts, latexsym, fancyhdr, longtable, ulem, mathtools}
\usepackage{epsfig, graphicx, pstricks,pst-grad,pst-text,tikz,colortbl}
\usepackage{epsf}
\usepackage{graphicx, color}
\usepackage{float}
\usepackage[rflt]{floatflt}
\usepackage{amsfonts}
\usepackage{latexsym,enumitem}
\usetikzlibrary{fit,matrix,positioning}
\usepackage{pdflscape}
\usetikzlibrary{decorations.pathreplacing}
\usepackage[most]{tcolorbox}
\usepackage{lipsum}
\usetikzlibrary{decorations.markings}
\usetikzlibrary{patterns}

\usepackage{mathrsfs}

\newcommand{\ldownarrow}{\Big\downarrow}

\newcommand{\precdot}{\prec\mathrel{\mkern-5mu}\mathrel{\cdot}}

\newtheorem{theorem}{Theorem}[section]
\newtheorem{lemma}[theorem]{Lemma}
\newtheorem{corollary}[theorem]{Corollary}

\newtheorem{example}[theorem]{Example}

\newtheorem{prop}[theorem]{Proposition}
\newtheorem*{theorem*}{Theorem}
\newtheorem{remark}[theorem]{Remark}

\usepackage{xargs}
\usepackage[colorinlistoftodos,prependcaption,textsize=tiny,linecolor=red,backgroundcolor=red!25,bordercolor=red]{todonotes}
\usepackage[backend=bibtex]{biblatex}
\begin{document}

\title{Ghost Kohnert posets}

\author[*]{Kelsey Hanser}
\author[*]{Nicholas Mayers}

\affil[*]{Department of Mathematics, North Carolina State University, Raleigh, NC, 27605}

\maketitle

\bigskip
\begin{abstract} 
\noindent
Recently, Pan and Yu showed that Lascoux polynomials can be defined in terms of certain collections of diagrams consisting of unit cells arranged in the first quadrant. Starting from certain initial diagrams, one forms a finite set of diagrams by applying two types of moves: Kohnert and ghost moves. Both moves cause at most one cell to move to a lower row with ghost moves leaving a new ``ghost cell" in its place. Each diagram formed in this way defines a monomial in the associated Lascoux polynomial. Restricting attention to diagrams formed by applying sequences of only Kohnert moves in the definition of Lascoux polynomials, one obtains the family of key polynomials. Recent articles have considered a poset structure on the collections of diagrams formed when one uses only Kohnert moves. In general, these posets are not ``well-behaved," not usually having desirable poset properties. Here, as an intermediate step to studying the analogous posets associated with Lascoux polynomials, we consider the posets formed by restricting attention to those diagrams formed by using only ghost moves. Unlike in the case of Kohnert posets, we show that such ``ghost Kohnert posets" are always ranked join semi-lattices. In addition, we establish a necessary condition for when ghost Kohnert posets are bounded and, consequently, lattices.
\end{abstract}

\section{Introduction}

In recent work \cite{AssafSchu,Kohnert,Pan1,Winkel2,Winkel1}, it has been found that many families of polynomials arising in algebraic combinatorics can be viewed as generating polynomials for certain collections of ``diagrams." Such families include Schubert \cite{LS82a}, key \cite{Dem74a,Dem74b}, and Lascoux polynomials \cite{Las04}. Within this framework, diagrams are collections of unit cells (some of which may be decorated) arranged in the first quadrant (see Figures~\ref{fig:dandm} and~\ref{fig:gm}) which are encoded by (signed) monomials where the power of $x_i$ is the number of cells in row $i$. For the polynomials known to have such definitions currently, the encoded diagrams are formed by starting from an initial ``seed" diagram and applying sequences of appropriately defined moves, i.e., Kohnert or K-Kohnert moves (see below as well as Section~\ref{sec:dagm}). Such moves effect the position of a single cell in the diagram and can additionally add decorated cells to the diagram as well. Here, we study a family of posets on decorated diagrams which arise within this framework.

Definitions for families of polynomials in terms of diagrams and moves, as described above, originated in the thesis of Kohnert \cite{Kohnert} where he showed that key polynomials admit such a definition and conjectured the same for Schubert polynomials; this conjecture was eventually proven by multiple authors \cite{AssafSchu,Winkel2,Winkel1}. The families of key and Schubert polynomials both form bases of the polynomial ring $\mathbb{Z}[x_1,x_2,\hdots]$ which consist of homogeneous polynomials that are indexed by weak compositions in the case of keys and permutations in the case of Schuberts. Given a weak composition $\alpha\in\mathbb{Z}^n_{\ge 0}$ and a permutation $w\in S_n$, we denote the associated key and Schubert polynomials by $\mathcal{K}_\alpha$ and $\mathfrak{S}_w$, respectively. Moreover, we associate with $\mathcal{K}_\alpha$ a corresponding ``key diagram," denoted $\mathbb{D}(\alpha)$, and with $\mathfrak{S}_w$ a corresponding ``Rothe diagram," denoted $\mathbb{D}(w)$. See Figure~\ref{fig:dandm} (a) for $\mathbb{D}(\alpha)$ with the weak composition $\alpha=(1,1,3,2)$ and Figure~\ref{fig:dandm} (d) for $\mathbb{D}(w)$ with the permutation $w=[4,2,5,3,1]$. For both key and Schubert polynomials, one starts from the associated diagram $D$ and forms a collection of diagrams, denoted $\mathrm{KD}(D)$, consisting of $D$ along with all those diagrams which can be formed from $D$ by applying sequences of ``Kohnert moves." Given a diagram $D$, applying a Kohnert move to a given nonempty row $r$ of $D$ causes the rightmost cell in row $r$ of $D$ to move to the first empty position below and in the same column if it exists; otherwise, the Kohnert move does nothing. In Figure~\ref{fig:dandm} we illustrate (a) $\mathbb{D}(\alpha)$ with $\alpha=(1,1,3,2)$, (b) the diagram formed from $\mathbb{D}(\alpha)$ by applying a Kohnert move at row 3, (c) the diagram formed from $\mathbb{D}(\alpha)$ by applying a Kohnert move at row 4, and (d) $\mathbb{D}(w)$ with $w=[4,2,5,3,1]\in S_5$. Note that applying a Kohnert move at row 1 or 2 of the diagram $\mathbb{D}(\alpha)$ would do nothing. 
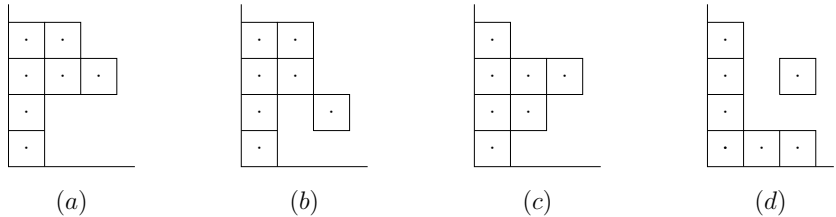
\begin{figure}[H]
    \centering
    $$\scalebox{0.8}{\begin{tikzpicture}[scale=0.6]
    \node at (0.5, 0.5) {$\cdot$};
  \node at (0.5, 1.5) {$\cdot$};
  \node at (0.5, 2.5) {$\cdot$};
  \node at (0.5, 3.5) {$\cdot$};
  \node at (1.5, 2.5) {$\cdot$};
  \node at (2.5, 2.5) {$\cdot$};
  \node at (1.5, 3.5) {$\cdot$};
  
  \draw (0,4.5)--(0,0)--(3.5,0);
  \draw (1,0)--(1,2)--(2,2)--(2,3)--(1,3)--(1,4)--(0,4);
  \draw (0,1)--(1,1);
  \draw (0,2)--(1,2)--(1,3)--(0,3);
  \draw (2,2)--(3,2)--(3,3)--(2,3);
  \draw (1,4)--(2,4)--(2,3);
  \node at (1.75, -1) {\large $(a)$};
\end{tikzpicture}}\quad\quad\quad\quad \scalebox{0.8}{\begin{tikzpicture}[scale=0.6]
    \node at (0.5, 0.5) {$\cdot$};
  \node at (0.5, 1.5) {$\cdot$};
  \node at (0.5, 2.5) {$\cdot$};
  \node at (0.5, 3.5) {$\cdot$};
  \node at (1.5, 2.5) {$\cdot$};
  \node at (2.5, 1.5) {$\cdot$};
  \node at (1.5, 3.5) {$\cdot$};
  
  \draw (0,4.5)--(0,0)--(3.5,0);
  \draw (1,0)--(1,2)--(2,2)--(2,3)--(1,3)--(1,4)--(0,4);
  \draw (0,1)--(1,1);
  \draw (0,2)--(1,2)--(1,3)--(0,3);
  \draw (2,1)--(3,1)--(3,2)--(2,2)--(2,1);
  \draw (1,4)--(2,4)--(2,3);
  \node at (1.75, -1) {\large $(b)$};
\end{tikzpicture}}\quad\quad\quad\quad \scalebox{0.8}{\begin{tikzpicture}[scale=0.6]
    \node at (0.5, 0.5) {$\cdot$};
  \node at (0.5, 1.5) {$\cdot$};
  \node at (0.5, 2.5) {$\cdot$};
  \node at (0.5, 3.5) {$\cdot$};
  \node at (1.5, 2.5) {$\cdot$};
  \node at (2.5, 2.5) {$\cdot$};
  \node at (1.5, 1.5) {$\cdot$};
  
  \draw (0,4.5)--(0,0)--(3.5,0);
  \draw (1,0)--(1,2)--(2,2)--(2,3)--(1,3)--(1,4)--(0,4);
  \draw (0,1)--(1,1);
  \draw (0,2)--(1,2)--(1,3)--(0,3);
  \draw (2,2)--(3,2)--(3,3)--(2,3);
  \draw (1,1)--(2,1)--(2,2);
  \node at (1.75, -1) {\large $(c)$};
\end{tikzpicture}}\quad\quad\quad\quad\scalebox{0.8}{\begin{tikzpicture}[scale=0.6]
    \node at (0.5, 0.5) {$\cdot$};
  \node at (0.5, 0.5) {$\cdot$};
  \node at (1.5, 0.5) {$\cdot$};
  \node at (2.5, 0.5) {$\cdot$};
  \node at (2.5, 2.5) {$\cdot$};
  \node at (0.5, 1.5) {$\cdot$};
  \node at (0.5, 2.5) {$\cdot$};
  \node at (0.5, 3.5) {$\cdot$};
  \draw (0,4.5)--(0,0)--(3.5,0);
  \draw (3,0)--(3,1)--(1,1)--(1,4)--(0,4);
  \draw (2,0)--(2,1);
  \draw (1,0)--(1,1)--(0,1);
  \draw (0,2)--(1,2);
  \draw (0,3)--(1,3);
  \draw (2,2)--(3,2)--(3,3)--(2,3)--(2,2);
  \node at (1.75, -1) {\large $(d)$};
\end{tikzpicture}}$$
    \caption{Diagrams and moves}
    \label{fig:dandm}
\end{figure}
\noindent
Using these notions, given a weak composition $\alpha$ and a permutation $w$, it was proven in \cite{Kohnert} and \cite{AssafSchu,Winkel2,Winkel1} that $$\mathcal{K}_\alpha=\sum_{D\in \mathrm{KD}(\mathbb{D}(\alpha))}\mathrm{wt}(D)\quad\quad\text{and}\quad\quad \mathfrak{S}_w=\sum_{D\in \mathrm{KD}(\mathbb{D}(w))}\mathrm{wt}(D),$$ respectively, where $\mathrm{wt}(D)=\prod_{i\ge 1}x_i^{n_i}$ with $n_i$ the number of cells in row $i$ of $D$.

Motivated by the results above concerning key and Schubert polynomials, Assaf and Searles \cite{Kohnert} define the \textit{Kohnert polynomial} of an arbitrary diagram $D$ by $$\mathfrak{K}_D=\sum_{T\in \mathrm{KD}(D)}\mathrm{wt}(T)$$ where $w(T)$ is defined as above. Now, in studying Kohnert polynomials, Assaf \cite{KP2} associates to each such polynomial $\mathfrak{K}_D$ a natural poset structure with underlying set $\mathrm{KD}(D)$ (see also \cite{KP3}). In particular, given $D_1,D_2\in \mathrm{KD}(D)$, one says that $D_2$ is smaller than $D_1$ if the smaller can be formed from the larger by applying some sequence of Kohnert moves. Moreover, in \cite{KP2} Assaf notes that such ``Kohnert posets" are not well-behaved in general, not usually having desirable poset properties; included among such misbehaved posets are some associated with both key and Schubert polynomials. As a result, various authors have attempted to characterize when Kohnert posets do have certain desirable poset properties and, in some cases, determine algebraic consequences \cite{KPoset3,KPoset1,KPoset2}; that is, relationships between desirable poset properties and properties of the associated Kohnert polynomial.

In another direction, an extension of Kohnert's definition for key and Schubert polynomials has been established for Lascoux polynomials \cite{Pan1,C1} and conjectured for Grothendieck  polynomials \cite{C2}. Here, we will be concerned with the work related to Lascoux polynomials \cite{Las04}, which are a class of nonhomogeneous polynomials indexed by weak compositions. As in the case of key and Schubert polynomials, Lascoux polynomials form a basis of the polynomial ring $\mathbb{Z}[x_1,x_2,\hdots]$. To provide the Kohnert-like definition of Lascoux polynomials established in \cite{Pan1}, we need to define a new type of move, called a ``ghost move." When applying a ghost move to a given row of a diagram $D$, one proceed exactly as when applying a Kohnert move except, when nontrivial, a new ``ghost cell" is left in place of the rightmost cell in the given row. Such ghost cells are decorated with an $\bigtimes$ in illustrations, are fixed by both Kohnert and ghost moves, and prevent cells strictly above from moving to rows strictly below. Letting $D$ be the diagram illustrated in Figure~\ref{fig:dandm} (a), in Figure~\ref{fig:gm} we illustrate (a) the diagram formed from $D$ by applying a ghost move at row 3 and (b) the diagram formed from $D$ by applying a ghost move at row 4.
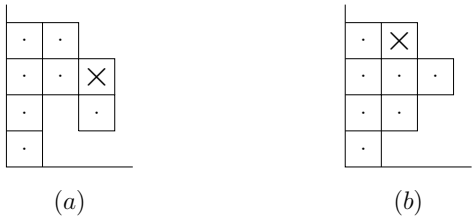
\begin{figure}[H]
    \centering
    $$\scalebox{0.8}{\begin{tikzpicture}[scale=0.6]
    \node at (0.5, 0.5) {$\cdot$};
  \node at (0.5, 1.5) {$\cdot$};
  \node at (0.5, 2.5) {$\cdot$};
  \node at (0.5, 3.5) {$\cdot$};
  \node at (1.5, 2.5) {$\cdot$};
  \node at (2.5, 1.5) {$\cdot$};
   \node at (2.5, 2.5) {$\bigtimes$};
  \node at (1.5, 3.5) {$\cdot$};
  
  \draw (0,4.5)--(0,0)--(3.5,0);
  \draw (2,3)--(3,3)--(3,2);
  \draw (1,0)--(1,2)--(2,2)--(2,3)--(1,3)--(1,4)--(0,4);
  \draw (0,1)--(1,1);
  \draw (0,2)--(1,2)--(1,3)--(0,3);
  \draw (2,1)--(3,1)--(3,2)--(2,2)--(2,1);
  \draw (1,4)--(2,4)--(2,3);
  \node at (1.75, -1) {\large $(a)$};
\end{tikzpicture}}\quad\quad\quad\quad\quad\quad\quad\quad \scalebox{0.8}{\begin{tikzpicture}[scale=0.6]
    \node at (0.5, 0.5) {$\cdot$};
  \node at (0.5, 1.5) {$\cdot$};
  \node at (0.5, 2.5) {$\cdot$};
  \node at (0.5, 3.5) {$\cdot$};
  \node at (1.5, 2.5) {$\cdot$};
  \node at (2.5, 2.5) {$\cdot$};
  \node at (1.5, 1.5) {$\cdot$};
  \node at (1.5, 3.5) {$\bigtimes$};
  
  \draw (0,4.5)--(0,0)--(3.5,0);
  \draw (1,0)--(1,2)--(2,2)--(2,3)--(1,3)--(1,4)--(0,4);
  \draw (0,1)--(1,1);
  \draw (0,2)--(1,2)--(1,3)--(0,3);
  \draw (2,2)--(3,2)--(3,3)--(2,3);
  \draw (1,1)--(2,1)--(2,2);
  \draw (2,3)--(2,4)--(1,4);
  \node at (1.75, -1) {\large $(b)$};
\end{tikzpicture}}$$
    \caption{Diagrams and moves}
    \label{fig:gm}
\end{figure}
\noindent
For a diagram $D$, we let $\mathrm{KKD}(D)$ denote the collection of diagrams consisting of $D$ along with all those diagrams that can be formed from $D$ by applying sequences of Kohnert and ghost moves. With the above notions in hand, given a weak composition $\alpha\in\mathbb{Z}_{\ge 0}^n$, in \cite{Pan1} it is shown that the associated Lascoux polynomial is equal to $$\mathcal{L}_\alpha=\sum_{D\in \mathrm{KKD}(\mathbb{D}(\alpha))}\mathrm{wt}^+(D)$$ where $\mathrm{wt}^+(D)=(-1)^g \mathrm{wt}(D)$ with $g$ equal to the number of ghost cells contained in $D$.

As in the case of Kohnert polynomials and $\mathrm{KD}(D)$, it is straightforward to associate a poset structure on the underlying set of diagrams encoded by Lascoux polynomials. More generally, for an arbitrary diagram $D$ and $D_1,D_2\in \mathrm{KKD}(D)$, we say that $D_2$ is less than $D_1$ if $D_2$ can be formed from $D_1$ by applying some sequence of Kohnert and ghost moves. Here, we initiate an investigation into this family of posets. In particular, we focus on a special subposet of such posets. Note that restricting attention to $D$ and all those diagrams which can be formed from $D$ by applying sequences of only Kohnert moves, we obtain the Kohnert poset associated with $D$. In this paper, we consider the subposet of $\mathrm{KKD}(D)$ corresponding to the other extreme; that is, we consider the subposet formed by restricting to $D$ along with all diagrams that can be formed from $D$ by applying sequences of only ghost moves. Unlike in the case of Kohnert posets, we find that these ``ghost Kohnert posets" do have some desirable properties in general. Specifically, we find that 
\begin{enumerate}
    \item[(1)] such posets are always ranked (Theorem~\ref{thm:ranked});
    \item[(2)] such posets are always join-semilattices (Theorem~\ref{thm:jsl}); and
    \item[(3)] while such posets are not always bounded, we are able to establish a necessary condition for when such posets are bounded in terms of the initial seed diagram (Theorem~\ref{thm:bounded-increase}).\footnote{In \cite{KPoset1}, a sufficient condition for a Kohnert poset to be bounded is provided by Proposition 3.2. At the time of writing, to the best of the authors knowledge, no general necessary condition is known.}
\end{enumerate}
Note that, considering (2), the necessary condition of (3) also applies for ghost Kohnert posets to be lattices.

The remainder of this article is organized as follows. In Section~\ref{sec:prelim}, we cover the necessary preliminaries concerning diagrams and posets, including the formal definition of ghost Kohnert poset. Following this, in Section~\ref{sec:results} we establish our main results, showing that ghost Kohnert posets are always ranked and join-semilattices, as well as establishing a necessary condition for when they are bounded. Also included in this final section is a discussion of directions for future research.

\section{Preliminaries}\label{sec:prelim}

In this section, we cover the preliminaries required to define ghost Kohnert posets as well as the desirable poset properties that we will be concerned with in this article. Ongoing, for $n\in\mathbb{Z}$, we let $[n]:=\{1,2,\hdots,n\}$.

\subsection{Diagrams and Ghost Moves}\label{sec:dagm}

As mentioned in the introduction, we will be interested in a class of posets whose underlying sets consist of ``diagrams." In this paper, a \textbf{diagram} is an array of finitely many cells in $\mathbb{N}\times\mathbb{N}$, where some of the cells may be decorated with an $\bigtimes$ and called \textbf{ghost cells}. Example diagrams are illustrated in Figure~\ref{fig:diagram} (a) and (b) below. Such decorated diagrams can be defined by the set of row/column coordinates of the cells defining it, where non-ghost cells are denoted by ordered pairs of the form $(r,c)$ and ghost cells by ordered pairs of the form $\langle r,c\rangle$. Consequently, if a diagram $D$ contains a non-ghost (resp., ghost) cell in position $(r,c)$, then we write $(r,c)\in D$ (resp., $\langle r,c\rangle\in D$); otherwise, we write $(r,c)\notin D$ (resp., $\langle r,c\rangle\notin D$).

\begin{example}
    The diagrams $$D_1=\{(1,1),(1,4),(2,2),(2,3),(3,1),(3,2),(3,4),(4,1)\}\quad\text{and}\quad D_2=\{(1,3),\langle2,2\rangle,(2,3),(3,1),(3,2),\langle4,1\rangle\}$$ are illustrated in Figures~\ref{fig:diagram} \textup{(a)} and \textup{(b)}, respectively.
    
    \begin{figure}[H]
    \centering
    $$\scalebox{0.8}{\begin{tikzpicture}[scale=0.6]
    \node at (0.5, 0.5) {$\cdot$};
  \node at (0.5, 2.5) {$\cdot$};
  \node at (0.5, 3.5) {$\cdot$};
  \node at (1.5, 1.5) {$\cdot$};
  \node at (1.5, 2.5) {$\cdot$};
  \node at (2.5, 1.5) {$\cdot$};
  \node at (3.5, 0.5) {$\cdot$};
  \node at (3.5, 2.5) {$\cdot$};
  \draw (0,4.5)--(0,0)--(4.5,0);
  \draw (0,2)--(1,2)--(1,4)--(0,4);
  \draw (0,3)--(1,3);
  \draw (0,1)--(1,1)--(1,0);
  \draw (1,2)--(1,1)--(2,1)--(2,3)--(1,3);
  \draw (1,2)--(2,2);
  \draw (2,1)--(3,1)--(3,2)--(2,2);
  \draw (3,0)--(3,1)--(4,1)--(4,0);
  \draw (3,2)--(4,2)--(4,3)--(3,3)--(3,2);
  \node at (2, -1) {\large $(a)$};
\end{tikzpicture}}\quad\quad\quad\quad\quad\quad\quad\quad \scalebox{0.8}{\begin{tikzpicture}[scale=0.6]
\node at (0.5,3.5) {$\bigtimes$};
\node at (0.5,2.5) {$\cdot$};
\node at (1.5,1.5) {$\bigtimes$};
\node at (1.5,2.5) {$\cdot$};
\node at (2.5,0.5) {$\cdot$};
\node at (2.5,1.5) {$\cdot$};
    \draw (0,4.5)--(0,0)--(3.5,0);
  \draw (0,2)--(1,2)--(1,4)--(0,4);
  \draw (1,2)--(1,1)--(2,1)--(2,3)--(1,3);
  \draw (1,2)--(2,2);
  \draw (0,3)--(1,3);
  \draw (2,0)--(2,1)--(3,1)--(3,0);
  \draw (2,2)--(3,2)--(3,1);
  \node at (2, -1) {\large $(b)$};
\end{tikzpicture}}$$
    \caption{Diagram}
    \label{fig:diagram}
\end{figure}
\end{example}

Now, to each nonempty row of a diagram we can apply what is called a ``$K$-Kohnert move" defined as follows. Given a diagram $D$ and a nonempty row $r$ of $D$, to apply a $K$-Kohnert move at row $r$ of $D$, we first find $(r,c)$ or $\langle r,c\rangle\in D$ with $c$ maximal, i.e., the rightmost cell in row $r$ of $D$. If 
\begin{itemize}
    \item $\langle r,c\rangle\in D$ is the rightmost cell in row $r$ of $D$, i.e., the rightmost cell in row $r$ of $D$ is a ghost cell,
    \item there exists no $\widehat{r}<r$ such that $(\widehat{r},c)\notin D$, i.e., there are no empty positions below the rightmost cell in row $r$ of $D$, or
    \item there exists $\widehat{r}<r^*<r$ such that $(\widehat{r},c)\notin D$, $\langle r^*,c\rangle\in D$, and $(\tilde{r},c)$ or $\langle\tilde{r},c\rangle\in D$ for $\widehat{r}<\tilde{r}\neq r^*<r$, i.e., there exists a ghost cell between the rightmost cell in row $r$ of $D$ and the highest empty position below,
\end{itemize}
then the $K$-Kohnert move does nothing; otherwise, there are two choices: letting $\widehat{r}<r$ be maximal such that $(\widehat{r},c)\notin D$, either
\begin{itemize}
    \item[(1)] $D$ becomes $(D\backslash (r,c))\cup \{(\widehat{r},c)\}$, i.e., the rightmost cell in row $r$ of $D$ moves to the highest empty position below, or
    \item[(2)] $D$ becomes $(D\backslash (r,c))\cup \{(\widehat{r},c),\langle r,c\rangle\}$, i.e., the rightmost cell in row $r$ of $D$ moves to the highest empty position below and leaves a ghost cell in its original position.
\end{itemize}
$K$-Kohnert moves of the form (1) are called \textbf{Kohnert moves}, while those of the form (2) are called \textbf{ghost moves}. In this article, we will make use of only ghost moves; the notions of Kohnert and $K$-Kohnert moves are defined solely for the sake of discussing motivation for such definitions (see Remark~\ref{rem:Lascoux}). Ongoing, we denote the diagram formed by applying a ghost move to a diagram $D$ at row $r$ by $\mathcal{G}(D,r)$. To aid in expressing the effect of applying a ghost move, we make use of the following notation. If applying a ghost move at row $r$ of $D$ causes the cell in position $(r,c)\in D$ to move down to position $(\widehat{r},c)$ in forming the diagram $\widehat{D}$, then we write $$\widehat{D}=\mathcal{G}(D,r)=D\ldownarrow^{(r,c)}_{(\widehat{r},c)}\cup\{\langle r,c\rangle\}$$ and refer to $(r,c)$ as a \textbf{free cell} of $D$. For a cell $(r,c)\in D$, if $(r,c)\in\mathcal{G}(D,r)$, then we call $(r,c)$ a \textbf{blocked cell}. In addition, all ghost cells are referred to as blocked cells. 

\begin{example}
    Let $D$ be the diagram illustrated in Figure~\ref{fig:diagram} \textup(a\textup) and $T$ be the diagram illustrated in Figure~\ref{fig:diagram} \textup(b\textup). The diagrams $$\mathcal{G}(D,2)=D\ldownarrow^{(2,3)}_{(1,3)}\cup\{\langle 2,3\rangle\},\quad\quad \mathcal{G}(D,3)=D\ldownarrow^{(3,4)}_{(2,4)}\cup\{\langle 3,4\rangle\},$$ and $$\mathcal{G}(D,4)=D\ldownarrow^{(4,1)}_{(2,1)}\cup\{\langle 4,1\rangle\}$$ are illustrated in Figure~\ref{fig:moves} \textup(a\textup), \textup(b\textup), and \textup(c\textup), respectively. Note that $D=\mathcal{G}(D,1)$ and $T=\mathcal{G}(T,i)$ for $i\in [4]$.
    \begin{figure}[H]
        \centering
        $$\scalebox{0.8}{\begin{tikzpicture}[scale=0.6]
    \node at (0.5, 0.5) {$\cdot$};
  \node at (0.5, 2.5) {$\cdot$};
  \node at (0.5, 3.5) {$\cdot$};
  \node at (1.5, 1.5) {$\cdot$};
  \node at (1.5, 2.5) {$\cdot$};
  \node at (2.5, 0.5) {$\cdot$};
  \node at (2.5, 1.5) {$\bigtimes$};
  \node at (3.5, 0.5) {$\cdot$};
  \node at (3.5, 2.5) {$\cdot$};
  \draw (0,4.5)--(0,0)--(4.5,0);
  \draw (0,2)--(1,2)--(1,4)--(0,4);
  \draw (0,3)--(1,3);
  \draw (0,1)--(1,1)--(1,0);
  \draw (1,2)--(1,1)--(2,1)--(2,3)--(1,3);
  \draw (2,0)--(2,1);
  \draw (1,2)--(2,2);
  \draw (2,1)--(3,1)--(3,2)--(2,2);
  \draw (3,0)--(3,1)--(4,1)--(4,0);
  \draw (3,2)--(4,2)--(4,3)--(3,3)--(3,2);
  \node at (2, -1) {\large $(a)$};
\end{tikzpicture}}\quad\quad\quad\quad\quad\quad\quad\quad\scalebox{0.8}{\begin{tikzpicture}[scale=0.6]
    \node at (0.5, 0.5) {$\cdot$};
  \node at (0.5, 2.5) {$\cdot$};
  \node at (0.5, 3.5) {$\cdot$};
  \node at (1.5, 1.5) {$\cdot$};
  \node at (1.5, 2.5) {$\cdot$};
  \node at (2.5, 1.5) {$\cdot$};
  \node at (3.5, 0.5) {$\cdot$};
  \node at (3.5, 1.5) {$\cdot$};
  \node at (3.5, 2.5) {$\bigtimes$};
  \draw (0,4.5)--(0,0)--(4.5,0);
  \draw (0,2)--(1,2)--(1,4)--(0,4);
  \draw (0,3)--(1,3);
  \draw (0,1)--(1,1)--(1,0);
  \draw (1,2)--(1,1)--(2,1)--(2,3)--(1,3);
  \draw (4,1)--(4,2);
  \draw (1,2)--(2,2);
  \draw (2,1)--(3,1)--(3,2)--(2,2);
  \draw (3,0)--(3,1)--(4,1)--(4,0);
  \draw (3,2)--(4,2)--(4,3)--(3,3)--(3,2);
  \node at (2, -1) {\large $(b)$};
\end{tikzpicture}}\quad\quad\quad\quad\quad\quad\quad\quad\scalebox{0.8}{\begin{tikzpicture}[scale=0.6]
    \node at (0.5, 0.5) {$\cdot$};
    \node at (0.5, 1.5) {$\cdot$};
  \node at (0.5, 2.5) {$\cdot$};
  \node at (0.5, 3.5) {$\bigtimes$};
  \node at (1.5, 1.5) {$\cdot$};
  \node at (1.5, 2.5) {$\cdot$};
  \node at (2.5, 1.5) {$\cdot$};
  \node at (3.5, 0.5) {$\cdot$};
  \node at (3.5, 2.5) {$\cdot$};
  \draw (0,4.5)--(0,0)--(4.5,0);
  \draw (0,2)--(1,2)--(1,4)--(0,4);
  \draw (0,3)--(1,3);
  \draw (0,1)--(1,1)--(1,0);
  \draw (1,2)--(1,1)--(2,1)--(2,3)--(1,3);
  \draw (1,2)--(2,2);
  \draw (2,1)--(3,1)--(3,2)--(2,2);
  \draw (3,0)--(3,1)--(4,1)--(4,0);
  \draw (3,2)--(4,2)--(4,3)--(3,3)--(3,2);
  \node at (2, -1) {\large $(c)$};
\end{tikzpicture}}$$
        \caption{Ghost moves}
        \label{fig:moves}
    \end{figure}
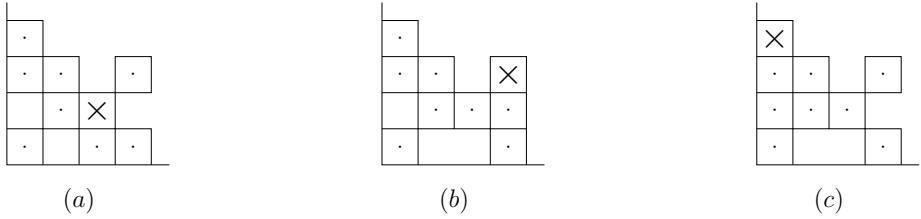
\end{example}

For a diagram $D$, we let
\begin{itemize}
    \item $\mathrm{KD}(D)$ denote the collection of diagrams formed from $D$ by applying sequences of only Kohnert moves,
    \item $\mathrm{GKD}(D)$ denote the collection of diagrams formed from $D$ by applying sequences of only ghost moves, and
    \item $\mathrm{KKD}(D)$ denote the collection of diagrams formed from $D$ by applying sequences of $K$-Kohnert moves.
\end{itemize}
Moreover, we define $$G(D)=\{\langle r,c\rangle\in D~|~r,c\in\mathbb{N}\},\quad \widehat{\mathrm{MaxG}}(D)=\max\{|G(T)|~|~T\in \mathrm{GKD}(D)\},$$ and $$R(D)=\{(r,c)\in D~|~(r,\tilde{c}),\langle r,\tilde{c}\rangle \notin D~\text{for}~\tilde{c}>c\};$$ that is, we denote by $G(D)$ the collection of ghost cells in $D$, $\widehat{\mathrm{MaxG}}(D)$ the maximum number of ghost cells contained within a diagram of $\mathrm{GKD}(D)$, and $R(D)$ the collection of rightmost cells in $D$.

\begin{remark}\label{rem:Lascoux}
    Both key and Lascoux polynomials form families of polynomials indexed by weak compositions, i.e., finite sequences of nonnegative integers. Given a weak composition $\alpha=(\alpha_1,\hdots,\alpha_n)\in\mathbb{Z}^n_{\ge 0}$, let $$\mathbb{D}(\alpha)=\bigcup_{i=1}^n\{(i,j)~|~1\le j\le \alpha_i\};$$ that is, $\mathbb{D}(\alpha)$ is the diagram with all cells left justified and with the numbers of cells per row given by $\alpha$. Denoting the key polynomial associated with a weak composition $\alpha$ by $\mathcal{K}_{\alpha}$, in \textup{\cite{Kohnert}} it is shown that $$\mathcal{K}_{\alpha}=\sum_{D\in \mathrm{KD}(\mathbb{D}(\alpha))}\mathrm{wt}(D),$$ where $\mathrm{wt}(D)=\prod_{r\ge 1}x_r^{|\{c~|~(r,c)\in D\}|}$; similarly, denoting the Lascoux polynomial associated with a weak composition $\alpha$ by $\mathcal{L}_{\alpha}$, in \textup{\cite{Pan1}} it is shown that $$\mathcal{L}_{\alpha}=\sum_{D\in \mathrm{KKD}(\mathbb{D}(\alpha))}\mathrm{wt}^+(D),$$ where $\mathrm{wt}^+(D)=\prod_{r\ge 1}(-1)^{|G(D)|}x_r^{|\{c~|~(r,c)~\text{or}~\langle r,c\rangle\in D\}|}.$
\end{remark}

\subsection{Poset theory}\label{subsec:KPos}

In this section, we cover the requisite preliminaries from the theory of posets. For more details, see \cite{EC1}.

A \textbf{finite poset} $(\mathcal{P},\preceq)$ consists of a finite underlying set $\mathcal{P}$ along with a binary relation $\preceq$ between the elements of $\mathcal{P}$ which is reflexive, anti-symmetric, and transitive. When no confusion will arise, we simply denote a poset $(\mathcal{P}, \preceq)$ by $\mathcal{P}$. 

Let $\mathcal{P}$ be a finite poset and take $x,y\in\mathcal{P}$. If $x\preceq y$ and $x\neq y$, then we call $x\preceq y$ a \textbf{strict relation} and write $x\prec y$. Ongoing, we let $\le$ and $<$ denote the relation and strict relation, respectively, corresponding to $\mathbb{Z}$ with its natural ordering. When $x\prec y$ and there exists no $z\in \mathcal{P}$ satisfying $x\prec z\prec y$, we refer to $x\prec y$ as a \textbf{covering relation} and write $x\precdot y$. Using covering relations, we define the \textbf{Hasse diagram} of $\mathcal{P}$ to be the graph whose vertices correspond to the elements of $\mathcal{P}$ and whose edges correspond to covering relations. We say $\mathcal{P}$ is \textbf{connected} if the Hasse diagram of $\mathcal{P}$ is a connected graph and say it is \textbf{disconnected} otherwise.

\begin{example}
    Let $\mathcal{P}=\{a,b,c,d,e,f\}$ be the poset defined by the relations $a\prec b,c\prec d\prec f$ and $c\prec e$. The Hasse diagram of $\mathcal{P}$ is illustrated in Figure~\ref{fig:Hasse}.
    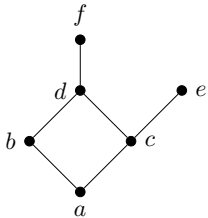
\begin{figure}[h]
        \centering
        $$\scalebox{0.9}{\begin{tikzpicture}
    \node (1) at (0,0) [circle, draw = black,fill = black, inner sep = 0.5mm, label=below:{$a$}] {};
    \node (2) at (-0.75,0.75) [circle, draw = black,fill = black, inner sep = 0.5mm, label=left:{$b$}] {};
    \node (3) at (0.75,0.75) [circle, draw = black,fill = black, inner sep = 0.5mm, label=right:{$c$}] {};
    \node (4) at (0,1.5) [circle, draw = black,fill = black, inner sep = 0.5mm, label=left:{$d$}] {};
    \node (5) at (1.5,1.5) [circle, draw = black,fill = black, inner sep = 0.5mm, label=right:{$e$}] {};
    \node (6) at (0,2.25) [circle, draw = black,fill = black, inner sep = 0.5mm, label=above:{$f$}] {};
    \draw (1)--(2)--(4)--(3)--(1);
    \draw (4)--(6);
    \draw (3)--(5);
\end{tikzpicture}}$$
        \caption{Hasse diagram}
        \label{fig:Hasse}
    \end{figure}
\end{example}

In this paper, among our posets of interest, we aim to identify those with certain desirable poset properties. The poset properties of interest are as follows. First, referring to $x\in\mathcal{P}$ as a \textbf{minimal element} (resp., \textbf{maximal element}) if there exists no $z\in\mathcal{P}$ such that $z\prec x$ (resp., $z\succ x$), we say that $\mathcal{P}$ is \textbf{bounded} provided it has both a unique minimal and unique maximal element. Next, we say that $\mathcal{P}$ is \textbf{ranked} if there exists a \textbf{rank function} $\rho:\mathcal{P}\to\mathbb{Z}_{\ge 0}$ satisfying 
\begin{enumerate}
    \item if $x\prec y$, then $\rho(x)<\rho(y)$; and
    \item if $x$ is covered by $y$, then $\rho(y)=\rho(x)+1$.
\end{enumerate}
    If a poset is ranked, then we can arrange its Hasse diagram into ``levels'' or ``ranks'' defined by the values of the associated rank function. Finally, we say that $\mathcal{P}$ is a \textbf{lattice} if between any pair of elements $x,y\in\mathcal{P}$, there exists a least upper bound, called a \textbf{join} and denoted $x\vee y$, and a greatest lower bound, called a \textbf{meet} and denoted $x\wedge y$. In the case that between any pair of elements $x,y\in\mathcal{P}$, there exists a join but not necessarily a meet, we refer to $\mathcal{P}$ as a \textbf{join-semilattice}.


\begin{example}
    Let $\mathcal{P}_1$, $\mathcal{P}_2$, and $\mathcal{P}_3$ denote the posets whose \textup(unlabeled\textup) Hasse diagrams are illustrated in Figure~\ref{fig:posetproperties} $(a)$, $(b)$, and $(c)$, respectively. Then we have that $\mathcal{P}_1$ is bounded, not ranked, and is a lattice; $\mathcal{P}_2$ is not bounded, is ranked, is a join-semilattice, but is not a lattice; and $\mathcal{P}_3$ is not bounded, is not a lattice nor a join-semilattice, but is ranked.
    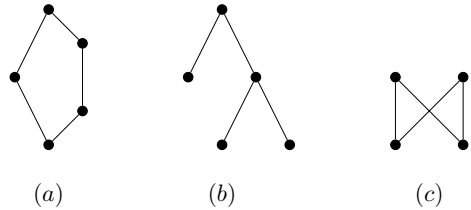
\begin{figure}[h]
        \centering
        $$\scalebox{0.9}{\begin{tikzpicture}
    \node (1) at (0,0) [circle, draw = black,fill = black, inner sep = 0.5mm] {};
    \node (2) at (-0.5,1) [circle, draw = black,fill = black, inner sep = 0.5mm] {};
    \node (3) at (0.5,0.5) [circle, draw = black,fill = black, inner sep = 0.5mm] {};
    \node (4) at (0.5,1.5) [circle, draw = black,fill = black, inner sep = 0.5mm] {};
    \node (5) at (0,2) [circle, draw = black,fill = black, inner sep = 0.5mm] {};
    \draw (1)--(2)--(5)--(4)--(3)--(1);
    \node at (0,-0.75) {$(a)$};
\end{tikzpicture}\quad\quad\quad\quad \begin{tikzpicture}
    \node (1) at (0,0) [circle, draw = black,fill = black, inner sep = 0.5mm] {};
    \node (2) at (1,0) [circle, draw = black,fill = black, inner sep = 0.5mm] {};
    \node (3) at (-0.5,1) [circle, draw = black,fill = black, inner sep = 0.5mm] {};
    \node (4) at (0.5,1) [circle, draw = black,fill = black, inner sep = 0.5mm] {};
    \node (5) at (0,2) [circle, draw = black,fill = black, inner sep = 0.5mm] {};
    \draw (1)--(4)--(2);
    \draw (4)--(5)--(3);
    \node at (0,-0.75) {$(b)$};
\end{tikzpicture}\quad\quad\quad\quad \begin{tikzpicture}
    \node (1) at (0,0) [circle, draw = black,fill = black, inner sep = 0.5mm] {};
    \node (2) at (1,0) [circle, draw = black,fill = black, inner sep = 0.5mm] {};
    \node (3) at (0,1) [circle, draw = black,fill = black, inner sep = 0.5mm] {};
    \node (4) at (1,1) [circle, draw = black,fill = black, inner sep = 0.5mm] {};
    \draw (1)--(3)--(2)--(4)--(1);
    \node at (0.5,-0.75) {$(c)$};
\end{tikzpicture}}$$
        \caption{Example posets}
        \label{fig:posetproperties}
    \end{figure}
\end{example}



We are now ready to formally define the ghost Kohnert poset associated with a diagram $D$. Given a diagram $D$, we define the \textbf{ghost Kohnert poset} $\mathcal{P}_G(D)$ to be the poset with underlying set $\mathrm{GKD}(D)$ and the following order. Given $D_1,D_2\in \mathrm{GKD}(D)$, we say $D_2\prec D_1$ if $D_2$ can be obtained from $D_1$ by applying some sequence of ghost moves. Note that by definition these posets are always connected.

\begin{example}\label{ex:poset}
    Letting $D$ be the diagram of Figure~\ref{fig:diagram} \textup(a\textup), the Hasse diagram of $\mathcal{P}_G(D)$ is illustrated in Figure~\ref{fig:GKP}.
    \begin{figure}[H]
        \centering
        $$\scalebox{0.7}{\begin{tikzpicture}
            \node (1) at (0,0) {\scalebox{0.7}{\begin{tikzpicture}[scale=0.6]
    \node at (0.5, 0.5) {$\cdot$};
    \node at (0.5, 1.5) {$\cdot$};
  \node at (0.5, 2.5) {$\cdot$};
  \node at (0.5, 3.5) {$\bigtimes$};
  \node at (1.5, 1.5) {$\cdot$};
  \node at (1.5, 2.5) {$\cdot$};
  \node at (2.5, 0.5) {$\cdot$};
  \node at (2.5, 1.5) {$\bigtimes$};
  \node at (3.5, 0.5) {$\cdot$};
  \node at (3.5, 1.5) {$\cdot$};
  \node at (3.5, 2.5) {$\bigtimes$};
  \draw (0,4.5)--(0,0)--(4.5,0);
  \draw (0,2)--(1,2)--(1,4)--(0,4);
  \draw (0,3)--(1,3);
  \draw (0,1)--(1,1)--(1,0);
  \draw (1,2)--(1,1)--(2,1)--(2,3)--(1,3);
  \draw (2,0)--(2,1);
  \draw (1,2)--(2,2);
  \draw (2,1)--(3,1)--(3,2)--(2,2);
  \draw (3,0)--(3,1)--(4,1)--(4,0);
  \draw (3,2)--(4,2)--(4,3)--(3,3)--(3,2);
  \draw (4,1)--(4,2);
\end{tikzpicture}}};
\node (2) at (-2,3) {\scalebox{0.7}{\begin{tikzpicture}[scale=0.6]
    \node at (0.5, 0.5) {$\cdot$};
  \node at (0.5, 2.5) {$\cdot$};
  \node at (0.5, 3.5) {$\cdot$};
  \node at (1.5, 1.5) {$\cdot$};
  \node at (1.5, 2.5) {$\cdot$};
  \node at (2.5, 0.5) {$\cdot$};
  \node at (2.5, 1.5) {$\bigtimes$};
  \node at (3.5, 0.5) {$\cdot$};
  \node at (3.5, 1.5) {$\cdot$};
  \node at (3.5, 2.5) {$\bigtimes$};
  \draw (0,4.5)--(0,0)--(4.5,0);
  \draw (0,2)--(1,2)--(1,4)--(0,4);
  \draw (0,3)--(1,3);
  \draw (0,1)--(1,1)--(1,0);
  \draw (1,2)--(1,1)--(2,1)--(2,3)--(1,3);
  \draw (2,0)--(2,1);
  \draw (1,2)--(2,2);
  \draw (2,1)--(3,1)--(3,2)--(2,2);
  \draw (3,0)--(3,1)--(4,1)--(4,0);
  \draw (3,2)--(4,2)--(4,3)--(3,3)--(3,2);
  \draw (4,1)--(4,2);
\end{tikzpicture}}};
\draw (1)--(2);
\node (3) at (2,3) {\scalebox{0.7}{\begin{tikzpicture}[scale=0.6]
    \node at (0.5, 0.5) {$\cdot$};
    \node at (0.5, 1.5) {$\cdot$};
  \node at (0.5, 2.5) {$\cdot$};
  \node at (0.5, 3.5) {$\bigtimes$};
  \node at (1.5, 1.5) {$\cdot$};
  \node at (1.5, 2.5) {$\cdot$};
  \node at (2.5, 0.5) {$\cdot$};
  \node at (2.5, 1.5) {$\bigtimes$};
  \node at (3.5, 0.5) {$\cdot$};
  \node at (3.5, 2.5) {$\cdot$};
  \draw (0,4.5)--(0,0)--(4.5,0);
  \draw (0,2)--(1,2)--(1,4)--(0,4);
  \draw (0,3)--(1,3);
  \draw (0,1)--(1,1)--(1,0);
  \draw (1,2)--(1,1)--(2,1)--(2,3)--(1,3);
  \draw (2,0)--(2,1);
  \draw (1,2)--(2,2);
  \draw (2,1)--(3,1)--(3,2)--(2,2);
  \draw (3,0)--(3,1)--(4,1)--(4,0);
  \draw (3,2)--(4,2)--(4,3)--(3,3)--(3,2);
\end{tikzpicture}}};
\draw (1)--(3);
\node (4) at (6,3) {\scalebox{0.7}{\begin{tikzpicture}[scale=0.6]
    \node at (0.5, 0.5) {$\cdot$};
    \node at (0.5, 1.5) {$\cdot$};
  \node at (0.5, 2.5) {$\cdot$};
  \node at (0.5, 3.5) {$\bigtimes$};
  \node at (1.5, 1.5) {$\cdot$};
  \node at (1.5, 2.5) {$\cdot$};
  \node at (2.5, 1.5) {$\cdot$};
  \node at (3.5, 0.5) {$\cdot$};
  \node at (3.5, 1.5) {$\cdot$};
  \node at (3.5, 2.5) {$\bigtimes$};
  \draw (0,4.5)--(0,0)--(4.5,0);
  \draw (0,2)--(1,2)--(1,4)--(0,4);
  \draw (0,3)--(1,3);
  \draw (0,1)--(1,1)--(1,0);
  \draw (1,2)--(1,1)--(2,1)--(2,3)--(1,3);
  \draw (1,2)--(2,2);
  \draw (2,1)--(3,1)--(3,2)--(2,2);
  \draw (3,0)--(3,1)--(4,1)--(4,0);
  \draw (3,2)--(4,2)--(4,3)--(3,3)--(3,2);
  \draw (4,1)--(4,2);
\end{tikzpicture}}};

\node (5) at (-2,6) {\scalebox{0.7}{\begin{tikzpicture}[scale=0.6]
    \node at (0.5, 0.5) {$\cdot$};
  \node at (0.5, 2.5) {$\cdot$};
  \node at (0.5, 3.5) {$\cdot$};
  \node at (1.5, 1.5) {$\cdot$};
  \node at (1.5, 2.5) {$\cdot$};
  \node at (2.5, 0.5) {$\cdot$};
  \node at (2.5, 1.5) {$\bigtimes$};
  \node at (3.5, 0.5) {$\cdot$};
  \node at (3.5, 2.5) {$\cdot$};
  \draw (0,4.5)--(0,0)--(4.5,0);
  \draw (0,2)--(1,2)--(1,4)--(0,4);
  \draw (0,3)--(1,3);
  \draw (0,1)--(1,1)--(1,0);
  \draw (1,2)--(1,1)--(2,1)--(2,3)--(1,3);
  \draw (2,0)--(2,1);
  \draw (1,2)--(2,2);
  \draw (2,1)--(3,1)--(3,2)--(2,2);
  \draw (3,0)--(3,1)--(4,1)--(4,0);
  \draw (3,2)--(4,2)--(4,3)--(3,3)--(3,2);
\end{tikzpicture}}};
\draw (2)--(5)--(3);

\node (6) at (6,6) {\scalebox{0.7}{\begin{tikzpicture}[scale=0.6]
    \node at (0.5, 0.5) {$\cdot$};
  \node at (0.5, 2.5) {$\cdot$};
  \node at (0.5, 3.5) {$\cdot$};
  \node at (1.5, 1.5) {$\cdot$};
  \node at (1.5, 2.5) {$\cdot$};
  \node at (2.5, 1.5) {$\cdot$};
  \node at (3.5, 0.5) {$\cdot$};
  \node at (3.5, 1.5) {$\cdot$};
  \node at (3.5, 2.5) {$\bigtimes$};
  \draw (0,4.5)--(0,0)--(4.5,0);
  \draw (0,2)--(1,2)--(1,4)--(0,4);
  \draw (0,3)--(1,3);
  \draw (0,1)--(1,1)--(1,0);
  \draw (1,2)--(1,1)--(2,1)--(2,3)--(1,3);
  \draw (1,2)--(2,2);
  \draw (2,1)--(3,1)--(3,2)--(2,2);
  \draw (3,0)--(3,1)--(4,1)--(4,0);
  \draw (3,2)--(4,2)--(4,3)--(3,3)--(3,2);
  \draw (4,1)--(4,2);
\end{tikzpicture}}};
\draw (6)--(4);

\node (7) at (2,6) {\scalebox{0.7}{\begin{tikzpicture}[scale=0.6]
    \node at (0.5, 0.5) {$\cdot$};
    \node at (0.5, 1.5) {$\cdot$};
  \node at (0.5, 2.5) {$\cdot$};
  \node at (0.5, 3.5) {$\bigtimes$};
  \node at (1.5, 1.5) {$\cdot$};
  \node at (1.5, 2.5) {$\cdot$};
  \node at (2.5, 1.5) {$\cdot$};
  \node at (3.5, 0.5) {$\cdot$};
  \node at (3.5, 2.5) {$\cdot$};
  \draw (0,4.5)--(0,0)--(4.5,0);
  \draw (0,2)--(1,2)--(1,4)--(0,4);
  \draw (0,3)--(1,3);
  \draw (0,1)--(1,1)--(1,0);
  \draw (1,2)--(1,1)--(2,1)--(2,3)--(1,3);
  \draw (1,2)--(2,2);
  \draw (2,1)--(3,1)--(3,2)--(2,2);
  \draw (3,0)--(3,1)--(4,1)--(4,0);
  \draw (3,2)--(4,2)--(4,3)--(3,3)--(3,2);
\end{tikzpicture}}};
\draw (3)--(7)--(4);

\node (8) at (2,9) {\scalebox{0.7}{\begin{tikzpicture}[scale=0.6]
    \node at (0.5, 0.5) {$\cdot$};
  \node at (0.5, 2.5) {$\cdot$};
  \node at (0.5, 3.5) {$\cdot$};
  \node at (1.5, 1.5) {$\cdot$};
  \node at (1.5, 2.5) {$\cdot$};
  \node at (2.5, 1.5) {$\cdot$};
  \node at (3.5, 0.5) {$\cdot$};
  \node at (3.5, 2.5) {$\cdot$};
  \draw (0,4.5)--(0,0)--(4.5,0);
  \draw (0,2)--(1,2)--(1,4)--(0,4);
  \draw (0,3)--(1,3);
  \draw (0,1)--(1,1)--(1,0);
  \draw (1,2)--(1,1)--(2,1)--(2,3)--(1,3);
  \draw (1,2)--(2,2);
  \draw (2,1)--(3,1)--(3,2)--(2,2);
  \draw (3,0)--(3,1)--(4,1)--(4,0);
  \draw (3,2)--(4,2)--(4,3)--(3,3)--(3,2);
\end{tikzpicture}}};
\draw (5)--(8)--(6);
\draw (8)--(7);
        \end{tikzpicture}}$$
        \caption{Ghost Kohnert poset}
        \label{fig:GKP}
    \end{figure}
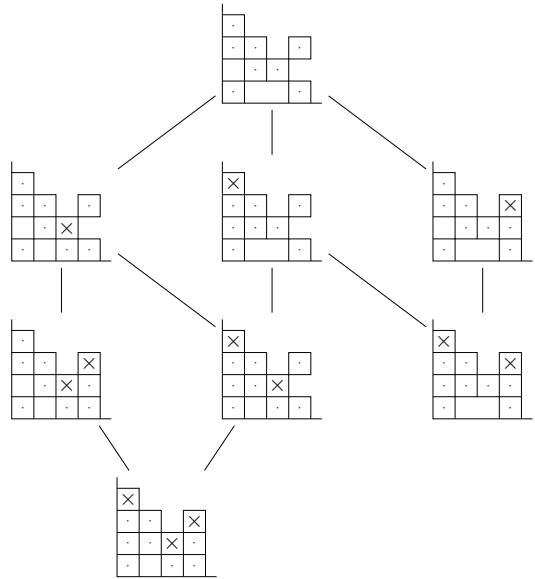
\end{example}

\section{Results}\label{sec:results}

In this section, we consider the structure of $\mathcal{P}_G(D)$. Specifically, we show that $\mathcal{P}_G(D)$ is a ranked join-semilattice for all choices of diagram $D$ with $|G(D)|=0$, i.e., that contains no ghost cells. Moreover, we establish a necessary condition for when $\mathcal{P}_G(D)$ is bounded in terms of $D$. 

\subsection{Ranked}

In this subsection, we show that $\mathcal{P}_G(D)$ is ranked for all choices of diagram $D$ with $|G(D)|=0$. To do so, we identify a natural rank function.

Given a diagram $D$ with $|G(D)|=0$ and $D_1,D_2\in\mathcal{P}_G(D)$ satisfying $D_2\precdot D_1$, considering the definition of our ordering, we have that $D_2$ can be formed from $D_1$ by applying a single ghost move. As a result, $D_2$ contains exactly one more ghost cell than $D_1$, i.e., $|G(D_1)|+1=|G(D_2)|$. Consequently, $|G(T)|$ for $T\in\mathcal{P}_G(D)$ starts at 0 for $T=D$ and increases as one moves downward in the Hasse diagram. Thus, the number of ghost cells contained in a given diagram would form a good candidate for a rank function of $\mathcal{P}_G(D)$ if all the inequalities in the definition of rank function were reversed. To fix this, taking $M=\widehat{\mathrm{MaxG}}(D)$, we instead consider the value $\rho(T)=M-|G(T)|$ for $T\in \mathcal{P}_G(D)$.

\begin{example}
    Let $D_1$ and $D_2$ be the diagrams contained in the ghost Kohnert poset $\mathcal{P}_G(D)$ of Example~\ref{ex:poset} which are illustrated in Figure~\ref{fig:ranked} $(a)$ and $(b)$, respectively. Computing, we have that $\rho(D_1)=3-3=0$ and $\rho(D_2)=3-2=1$.
    \begin{figure}[H]
        \centering
        $$\scalebox{0.8}{\begin{tikzpicture}[scale=0.6]
    \node at (0.5, 0.5) {$\cdot$};
    \node at (0.5, 1.5) {$\cdot$};
  \node at (0.5, 2.5) {$\cdot$};
  \node at (0.5, 3.5) {$\bigtimes$};
  \node at (1.5, 1.5) {$\cdot$};
  \node at (1.5, 2.5) {$\cdot$};
  \node at (2.5, 0.5) {$\cdot$};
  \node at (2.5, 1.5) {$\bigtimes$};
  \node at (3.5, 0.5) {$\cdot$};
  \node at (3.5, 1.5) {$\cdot$};
  \node at (3.5, 2.5) {$\bigtimes$};
  \draw (0,4.5)--(0,0)--(4.5,0);
  \draw (0,2)--(1,2)--(1,4)--(0,4);
  \draw (0,3)--(1,3);
  \draw (0,1)--(1,1)--(1,0);
  \draw (1,2)--(1,1)--(2,1)--(2,3)--(1,3);
  \draw (2,0)--(2,1);
  \draw (1,2)--(2,2);
  \draw (2,1)--(3,1)--(3,2)--(2,2);
  \draw (3,0)--(3,1)--(4,1)--(4,0);
  \draw (3,2)--(4,2)--(4,3)--(3,3)--(3,2);
  \draw (4,1)--(4,2);
  \node at (2,-1) {\large $(a)$};
\end{tikzpicture}}\quad\quad\quad\quad \scalebox{0.8}{\begin{tikzpicture}[scale=0.6]
    \node at (0.5, 0.5) {$\cdot$};
    \node at (0.5, 1.5) {$\cdot$};
  \node at (0.5, 2.5) {$\cdot$};
  \node at (0.5, 3.5) {$\bigtimes$};
  \node at (1.5, 1.5) {$\cdot$};
  \node at (1.5, 2.5) {$\cdot$};
  \node at (2.5, 0.5) {$\cdot$};
  \node at (2.5, 1.5) {$\bigtimes$};
  \node at (3.5, 0.5) {$\cdot$};
  \node at (3.5, 2.5) {$\cdot$};
  \draw (0,4.5)--(0,0)--(4.5,0);
  \draw (0,2)--(1,2)--(1,4)--(0,4);
  \draw (0,3)--(1,3);
  \draw (0,1)--(1,1)--(1,0);
  \draw (1,2)--(1,1)--(2,1)--(2,3)--(1,3);
  \draw (2,0)--(2,1);
  \draw (1,2)--(2,2);
  \draw (2,1)--(3,1)--(3,2)--(2,2);
  \draw (3,0)--(3,1)--(4,1)--(4,0);
  \draw (3,2)--(4,2)--(4,3)--(3,3)--(3,2);
  \node at (2,-1) {\large $(b)$};
\end{tikzpicture}}$$
        \caption{$\rho(D)$}
        \label{fig:ranked}
    \end{figure}
\end{example}
 
\begin{theorem}\label{thm:ranked}
    For any diagram $D$ with $|G(D)|=0$, the poset $\mathcal{P}_G(D)$ is ranked. Moreover, if $M=\widehat{\mathrm{MaxG}}(D)$, then a rank function is given by $\rho(T)=M-|G(T)|$.
\end{theorem}
\begin{proof}
    Taking $D_1,D_2\in\mathcal{P}_G(D)$, if $D_2\prec D_1$, then $D_2$ can be formed from $D_1$ by applying a sequence of ghost moves. Consequently, $|G(D_1)|<|G(D_2)|$ so that $$\rho(D_2)=M-|G(D_2)|>M-|G(D_1)|=\rho(D_1).$$ Moreover, as noted above, if $D_2\precdot D_1$, then $|G(D_1)|=|G(D_2)|+1$ so that $$\rho(D_2)=M-|G(D_2)|=M-(|G(D_1)|-1)=(M-|G(D_1)|)+1=\rho(D_1)+1.$$ The result follows.
\end{proof}

\begin{corollary}\label{cor:cover}
    Let $D$ be a diagram with $|G(D)|=0$. Then $D_1,D_2\in\mathcal{P}_G(D)$ satisfy $D_2\precdot D_1$ if and only if there exists $r\in\mathbb{Z}_{>0}$ such that $D_2=\mathcal{G}(D_1,r)$.
\end{corollary}
\begin{proof}
    As noted above, $D_2\precdot D_1$ implies that $D_2$ can be formed from $D_1$ by applying a single ghost move, i.e., there exists $r\in\mathbb{Z}_{>0}$ such that $D_2=\mathcal{G}(D_1,r)$. On the other hand, assume that there exists $r\in\mathbb{Z}_{>0}$ such that $D_2=\mathcal{G}(D_1,r)$, but $D_1$ does not cover $D_2$ in $\mathcal{P}_G(D)$. Then there exists $T\in\mathcal{P}_G(D)$ such that $D_2\prec T\prec D_1$. Since $\mathcal{P}_G(D)$ is ranked by $\rho(\tilde{D})=M - |G(\tilde{D})|$ for $\tilde{D}\in\mathcal{P}_G(D)$, it follows that $|G(D_1)|<|G(T)|<|G(D_2)|$, i.e., $|G(D_2)|>|G(D_1)|+1$; but, by definition, $|G(D_2)|=|G(D_1)|+1$, which is a contradiction. Thus, no such $T\in\mathcal{P}_G(D)$ can exist so that $D_2\precdot D_1$.
\end{proof}

\subsection{Join-Semilattice}

Next, we show that $\mathcal{P}_G(D)$ is a join-semilattice for all choices of diagram $D$ with $|G(D)|=0$. To do so, we make use of Lemma~\ref{lem:nathan} below which follows from \textup{\cite[Lemma 9-2.5]{Nathan}}.

\begin{lemma}\label{lem:nathan}
    Let $P$ be a finite poset with a unique maximal element. Suppose that, for any $x$ and $y$ in $P$ such that $x$ and $y$ are covered by a common element $z$, either $\{x,y\}$ has no lower bound or the meet $x\wedge y$ exists. Then $P$ is a join-semilattice.
\end{lemma}

\begin{example}
    Let $D_1$, $D_2$, and $D_3$ be the diagrams contained in the ghost Kohnert poset $\mathcal{P}_G(D)$ of Example~\ref{ex:poset} which are illustrated in Figure~\ref{fig:joinsemi} $(a)$, $(b)$, and $(c)$, respectively. Note that $D_1$, $D_2$, and $D_3$ are covered by a common element, i.e., $D$ of Example~\ref{ex:poset}. While $D_1$ and $D_2$ as well as $D_2$ and $D_3$ have a meet, $D_1$ and $D_3$ have no common lower bound.
    \begin{figure}[h]
        \centering
        $$\scalebox{0.8}{\begin{tikzpicture}[scale=0.6]
    \node at (0.5, 0.5) {$\cdot$};
  \node at (0.5, 2.5) {$\cdot$};
  \node at (0.5, 3.5) {$\cdot$};
  \node at (1.5, 1.5) {$\cdot$};
  \node at (1.5, 2.5) {$\cdot$};
  \node at (2.5, 0.5) {$\cdot$};
  \node at (2.5, 1.5) {$\bigtimes$};
  \node at (3.5, 0.5) {$\cdot$};
  \node at (3.5, 2.5) {$\cdot$};
  \draw (0,4.5)--(0,0)--(4.5,0);
  \draw (0,2)--(1,2)--(1,4)--(0,4);
  \draw (0,3)--(1,3);
  \draw (0,1)--(1,1)--(1,0);
  \draw (1,2)--(1,1)--(2,1)--(2,3)--(1,3);
  \draw (2,0)--(2,1);
  \draw (1,2)--(2,2);
  \draw (2,1)--(3,1)--(3,2)--(2,2);
  \draw (3,0)--(3,1)--(4,1)--(4,0);
  \draw (3,2)--(4,2)--(4,3)--(3,3)--(3,2);
  \node at (2,-1) {\large $(a)$};
\end{tikzpicture}}\quad\quad\quad\quad\scalebox{0.8}{\begin{tikzpicture}[scale=0.6]
    \node at (0.5, 0.5) {$\cdot$};
    \node at (0.5, 1.5) {$\cdot$};
  \node at (0.5, 2.5) {$\cdot$};
  \node at (0.5, 3.5) {$\bigtimes$};
  \node at (1.5, 1.5) {$\cdot$};
  \node at (1.5, 2.5) {$\cdot$};
  \node at (2.5, 1.5) {$\cdot$};
  \node at (3.5, 0.5) {$\cdot$};
  \node at (3.5, 2.5) {$\cdot$};
  \draw (0,4.5)--(0,0)--(4.5,0);
  \draw (0,2)--(1,2)--(1,4)--(0,4);
  \draw (0,3)--(1,3);
  \draw (0,1)--(1,1)--(1,0);
  \draw (1,2)--(1,1)--(2,1)--(2,3)--(1,3);
  \draw (1,2)--(2,2);
  \draw (2,1)--(3,1)--(3,2)--(2,2);
  \draw (3,0)--(3,1)--(4,1)--(4,0);
  \draw (3,2)--(4,2)--(4,3)--(3,3)--(3,2);
  \node at (2,-1) {\large $(b)$};
\end{tikzpicture}}\quad\quad\quad\quad\scalebox{0.8}{\begin{tikzpicture}[scale=0.6]
    \node at (0.5, 0.5) {$\cdot$};
  \node at (0.5, 2.5) {$\cdot$};
  \node at (0.5, 3.5) {$\cdot$};
  \node at (1.5, 1.5) {$\cdot$};
  \node at (1.5, 2.5) {$\cdot$};
  \node at (2.5, 1.5) {$\cdot$};
  \node at (3.5, 0.5) {$\cdot$};
  \node at (3.5, 1.5) {$\cdot$};
  \node at (3.5, 2.5) {$\bigtimes$};
  \draw (0,4.5)--(0,0)--(4.5,0);
  \draw (0,2)--(1,2)--(1,4)--(0,4);
  \draw (0,3)--(1,3);
  \draw (0,1)--(1,1)--(1,0);
  \draw (1,2)--(1,1)--(2,1)--(2,3)--(1,3);
  \draw (1,2)--(2,2);
  \draw (2,1)--(3,1)--(3,2)--(2,2);
  \draw (3,0)--(3,1)--(4,1)--(4,0);
  \draw (3,2)--(4,2)--(4,3)--(3,3)--(3,2);
  \draw (4,1)--(4,2);
  \node at (2,-1) {\large $(c)$};
\end{tikzpicture}}$$
        \caption{Lemma~\ref{lem:nathan}}
        \label{fig:joinsemi}
    \end{figure}
\end{example}

Note that to apply Lemma~\ref{lem:nathan}, one needs to show that certain pairs of elements (those covered by a common element) in the poset of interest either have a meet or have no lower bound.  In Proposition~\ref{prop:meet} below, we identify a condition under which such pairs of elements in $\mathcal{P}_G(D)$ have a meet. For the proof of Proposition~\ref{prop:meet}, we require the following lemmas.

\begin{lemma}\label{lem:block}
    Let $D$ be a diagram with $|G(D)|=0$ and $T\in \mathrm{GKD}(D)$.
    \begin{enumerate}
        \item[$(a)$] If $(r,c)\in T$, then for all $\tilde{T}\prec T$ either $(r,c)\in\tilde{T}$ or $\langle r,c\rangle\in\tilde{T}$.
        \item[$(b)$] Suppose that that there exists $c,r_1,r_2\in\mathbb{Z}_{>0}$ such that $r_1<r_2$ and $\langle r_1,c\rangle,(\tilde{r},c)\in T$ for all $r_1<\tilde{r}\le r_2$. Then $(r_2,c)\in\tilde{T}$ for all $\tilde{T}\preceq T$. 
    \end{enumerate}
\end{lemma}
\begin{proof}
    (a) Follows immediately from the definition of ghost move.
    
    (b) Considering the definition of ghost move, since the ghost cell $\langle r_1,c\rangle\in T$ lies between the cells $(\tilde{r},c)\in T$ for $r_1<\tilde{r}\le r_2$ and any empty position below, it follows that for all $\tilde{T}\prec T$, we have $\langle r_1,c\rangle,(\tilde{r},c)\in \tilde{T}$ for $r_1<\tilde{r}\le r_2$. In particular, $(r_2,c)\in \tilde{T}$ for all $\tilde{T}\prec T$.
\end{proof}


\begin{lemma}\label{lem:meethelp}
    Let $D$ be a diagram with $|G(D)|=0$.
    Suppose that $D_0,D_1,D_2,D_3\in \mathrm{GKD}(D)$ satisfy $D_3\precdot D_i\precdot D_0$ for $i=1,2$. Then there exists $r_1,r_2,c_1,c_2,r_1^*,r_2^*\in\mathbb{N}$ for which $r_1\neq r_2$, $$D_1=\mathcal{G}(D_0,r_1)=D_0\ldownarrow^{(r_1,c_1)}_{(r^*_1,c_1)}\cup\{\langle r_1,c_1\rangle\},\quad\quad D_2=\mathcal{G}(D_0,r_2)=D_0\ldownarrow^{(r_2,c_2)}_{(r^*_2,c_2)}\cup\{\langle r_2,c_2\rangle\},$$ and $$D_3=\mathcal{G}(D_1,r_2)=D_1\ldownarrow^{(r_2,c_2)}_{(r^*_2,c_2)}\cup\{\langle r_2,c_2\rangle\}=\mathcal{G}(D_2,r_1)=D_2\ldownarrow^{(r_1,c_1)}_{(r^*_1,c_1)}\cup\{\langle r_1,c_1\rangle\}.$$
\end{lemma}
\begin{proof}
    Since $D_1,D_2\precdot D_0$, there exists $r_1,r_2,c_1,c_2,r_1^*,r_2^*\in\mathbb{N}$ for which $r_1\neq r_2$, $$D_1=\mathcal{G}(D_0,r_1)=D_0\ldownarrow^{(r_1,c_1)}_{(r^*_1,c_1)}\cup\{\langle r_1,c_1\rangle\},\quad\text{and}\quad D_2=\mathcal{G}(D_0,r_2)=D_0\ldownarrow^{(r_2,c_2)}_{(r^*_2,c_2)}\cup\{\langle r_2,c_2\rangle\};$$ note that, as a consequence, for $i=1$ or 2, we have that $(\tilde{r},c_i)\in D_0$ for $r_i^*<\tilde{r}\le r_i$. Now, since $D_3\precdot D_1,D_2$, it follows that $\langle r_1,c_1\rangle,\langle r_2,c_2\rangle\in D_3$. Consequently, $$D_3=\mathcal{G}(D_1,r_2)=D_1\ldownarrow^{(r_2,c_2)}_{(\widehat{r}_2,c_2)}\cup\{\langle r_2,c_2\rangle\}=\mathcal{G}(D_2,r_1)=D_2\ldownarrow^{(r_1,c_1)}_{(\widehat{r}_1,c_1)}\cup\{\langle r_1,c_1\rangle\};$$ note that this implies that $(r_1,c_1)$ (resp., $(r_2,c_2)$) is rightmost in row $r_1$ (resp., $r_2$) of $D_2$ (resp., $D_1$). It remains to show that $\widehat{r}_1=r_1^*$ and $\widehat{r}_2=r_2^*$. To see this, assume that $\widehat{r}_1\neq r_1^*$, the other case following via a symmetric argument. Since $(\tilde{r},c_1)\in D_0$ for $r_1^*< \tilde{r}\le r_1$ and $D_2\precdot D_0$, applying Lemma~\ref{lem:block} (a), it follows that $(\tilde{r},c_1)$ or $\langle \tilde{r},c_1\rangle\in D_2$ for all $r_1^*<\tilde{r}\le r_1$. Considering Lemma~\ref{lem:block} (b) and the fact that $\mathcal{G}(D_2,r_1)=D_3\neq D_2$, it must be the case that $(\tilde{r},c_1)\in D_2$ for all $r_1^*<\tilde{r}\le r_1$. Moreover, since $r_1^*\neq\widehat{r}_1$, it must also be the case that $(r_1^*,c_1)\in D_2$. Thus, as $$D_2=\mathcal{G}(D_0,r_2)=D_0\ldownarrow^{(r_2,c_2)}_{(r^*_2,c_2)}\cup\{\langle r_2,c_2\rangle\}$$ and position $(r_1^*,c_1)$ is empty in $D_0$, it follows that $c_1=c_2$, $r_1^*=r_2^*$; that is, $$D_2=\mathcal{G}(D_0,r_2)=D_0\ldownarrow^{(r_2,c_1)}_{(r^*_1,c_1)}\cup\{\langle r_2,c_1\rangle\}.$$ In addition, since $(\tilde{r},c_1)\in D_2$ for all $r_1^*<\tilde{r}\le r_1$, it must be the case that $r_2>r_1$. Consequently, $(\tilde{r},c_1)\in D_0$ for all $r_1^*<\tilde{r}\le r_2$; but this implies that $\langle r_1,c_1\rangle,(\tilde{r},c_1)\in D_1$ for all $r_1<\tilde{r}\le r_2$. Since $(r_2,c_1)$ is rightmost in row $r_2$ of $D_1$, as noted above, it follows that $\mathcal{G}(D_1,r_2)=D_1\neq D_3$, which is a contradiction. Therefore, $r_1^*=\widehat{r}_1$ and the result follows.
\end{proof}

\begin{prop}\label{prop:meet}
    Let $D$ be a diagram with $|G(D)|=0$.
    If $D_0,D_1,D_2,D_3\in \mathrm{GKD}(D)$ satisfy $D_3\precdot D_i\precdot D_0$ for $i=1,2$, then $D_3=D_1\wedge D_2$.
\end{prop}

\begin{proof}
    Since $D_3\precdot D_1,D_2$, we have that $D_3$ forms a lower bound of $D_1,D_2$. To show that $D_3=D_1\wedge D_2$, i.e., $D_3$ is the greatest lower bound of $D_1$ and $D_2$, we need to show that if $\tilde{D}\in \mathcal{P}_G(D)$ satisfies $\tilde{D}\prec D_1,D_2$, then $\tilde{D}\preceq D_3$. Applying Lemma~\ref{lem:meethelp}, we have that there exists $r_1,r_2,c_1,c_2,r_1^*,r_2^*\in\mathbb{N}$ for which $r_1\neq r_2$, $$D_1=\mathcal{G}(D_0,r_1)=D_0\ldownarrow^{(r_1,c_1)}_{(r^*_1,c_1)}\cup\{\langle r_1,c_1\rangle\},\quad\quad D_2=\mathcal{G}(D_0,r_2)=D_0\ldownarrow^{(r_2,c_2)}_{(r^*_2,c_2)}\cup\{\langle r_2,c_2\rangle\},$$ and $$D_3=\mathcal{G}(D_1,r_2)=D_1\ldownarrow^{(r_2,c_2)}_{(r^*_2,c_2)}\cup\{\langle r_2,c_2\rangle\}=\mathcal{G}(D_2,r_1)=D_2\ldownarrow^{(r_1,c_1)}_{(r^*_1,c_1)}\cup\{\langle r_1,c_1\rangle\}.$$ Note that as a consequence, we have that $(\tilde{r},c_2)\in D_1$ for $r_2^*<\tilde{r}\le r_2$.

Take $\widehat{D}\in\mathcal{P}_G(D)$ such that $\widehat{D}\prec D_1,D_2$. We need to show that $\widehat{D}\preceq D_3$. Since $\widehat{D}\prec D_1$, there exists a sequence of rows $\{\widehat{r}_i\}_{i=1}^n$ such that if $T_0=D_1$ and $T_i=\mathcal{G}(T_{i-1},\widehat{r}_i)$ for $i\in [n]$, then $T_n=\widehat{D}$. Now, since $\widehat{D}\prec D_2$, it follows that $\langle r_2,c_2\rangle\in\widehat{D}$. Consequently, since $\langle r_2,c_2\rangle\notin D_1$, there exists $j\in [n]$ such that $\widehat{r}_j=r_2$, $\langle r_2,c_2\rangle\notin T_{j-1}$, and $\langle r_2,c_2\rangle\in T_j$. We claim that $T_j\preceq D_3$. Note that, since $\widehat{D}\preceq T_j$, if we can establish the claim, then it will follow that $\widehat{D}\preceq D_3$, as desired. As a first step towards establishing the claim, we show that $T_j=T_{j-1}\ldownarrow^{(r_2,c_2)}_{(r_2^*,c_2)}\cup\{\langle r_2,c_2\rangle\}$; that is, applying a ghost move at row $r_2$ of $T_{j-1}$ has the same effect as applying a ghost move at row $r_2$ of $D_1$. Assume otherwise, i.e., 
\begin{equation}\label{eq:joinsemi1}
    T_j= T_{j-1}\ldownarrow^{(r_2,c_2)}_{(r'_2,c_2)}\cup\{\langle r_2,c_2\rangle\}
\end{equation}
for $r'_2\neq r_2^*$. Considering \eqref{eq:joinsemi1} and the fact that $(\tilde{r},c_2)\in D_1$ for $r_2^*<\tilde{r}\le r_2$, applying Lemma~\ref{lem:block} (a), it follows that $(\tilde{r},c_2)\in T_{j-1}$ for $r_2^*<\tilde{r}\le r_2$. Moreover, since we are assuming that $T_j\neq T_{j-1}\ldownarrow^{(r_2,c_2)}_{(r_2^*,c_2)}\cup\{\langle r_2,c_2\rangle\}$, it must be the case that $r_2'<r_2^*$ and $(r_2^*,c_2)\in T_{j-1}$. Thus, since the effect of ghost moves on non-ghost cells is to cause them move to lower rows within the same column, this implies that $$|\{(r,c_2)\in D_1~|~r\le r_2\}|<|\{(r,c_2)\in T_{j-1}~|~r\le r_2\}|\le |\{(r,c_2)\in \widehat{D}~|~r\le r_2\}|,$$ i.e., there are strictly more cells weakly below row $r_2$ in column $c_2$ of $\widehat{D}$ than there are weakly below row $r_2$ in column $c_2$ of $D_1$. However, as $D_2=(D_0\backslash\{(r_2,c_2)\})\cup\{\langle r_2,c_2\rangle, (r_2^*,c_2)\}$ and $\widehat{D}\prec D_2$, we have that $$|\{(r,c_2)\in D_0~|~r\le r_2\}|=|\{(r,c_2)\in D_2~|~r\le r_2\}|=|\{(r,c_2)\in \widehat{D}~|~r\le r_2\}|,$$ where the last equality follows since ghost moves cannot cause a non-ghost cell to move from above a ghost cell to below a ghost cell; but, since $D_1=(D_0\backslash\{(r_1,c_2)\})\cup\{\langle r_1,c_1\rangle,(r_1^*,c_1)\}$, we must also have that $$|\{(r,c_2)\in D_0~|~r\le r_2\}|\le|\{(r,c_2)\in D_1~|~r\le r_2\}|$$ so that $$|\{(r,c_2)\in \widehat{D}~|~r\le r_2\}|\le |\{(r,c_2)\in D_1~|~r\le r_2\}|<|\{(r,c_2)\in \widehat{D}~|~r\le r_2\}|,$$ which is a contradiction. Thus, $T_j=T_{j-1}\ldownarrow^{(r_2,c_2)}_{(r_2^*,c_2)}\cup\{\langle r_2,c_2\rangle\}$, as desired. Note that, considering Lemma~\ref{lem:block} (a), it must be the case that $(r^*_2,c_2),\langle r_2^*,c_2\rangle\notin T_i$ for $i\in [j-1]$.

Now, setting $T'_0=D_3$ and $T'_i=\mathcal{G}(T'_{i-1},\widehat{r}_i)$ for $i\in [j-1]$, we show that $T_i\backslash T_{i-1}=T'_i\backslash T'_{i-1}$, i.e., applying a ghost move at row $\widehat{r}_i$ of $T_{i-1}$ has the same effect as applying a ghost move at row $\widehat{r}_i$ of $T'_{i-1}$. Since $D_3\backslash D_2=T_j\backslash T_{j-1}$ as shown above, it will then follow that $T'_{j-1}=T_j$ and, consequently, $T_j\preceq D_3$, as claimed. Assume otherwise. Then there exists a least $j^*\in [j-1]$ such that $T_{j^*}\backslash T_{j^*-1}\neq T'_{j^*}\backslash T'_{j^*-1}$. Since $T_i\backslash T_{i-1}=T'_i\backslash T'_{i-1}$ for $i\in [j^*-1]$ and the diagrams $T'_0$ and $T_0$ differ only in positions $(r^*_2,c_2)$ and $(r_2,c_2)$ with $(r_2,c_2)\in T_0$, $(r_2^*,c_2),\langle r_2,c_2\rangle\notin T_0$, and $(r_2^*,c_2),\langle r_2,c_2\rangle\in T'_0$, it follows that $T_{j^*-1}$ and $T'_{j^*-1}$ differ only in positions $(r^*_2,c_2)$ and $(r_2,c_2)$ with $(r_2,c_2)\in T_{j^*-1}$, $(r^*_2,c_2),\langle r_2,c_2\rangle\notin T_{j^*-1}$, and $(r_2^*,c_2),\langle r_2,c_2\rangle\in T'_{j^*-1}$. Now, suppose that 
    \begin{equation}\label{eq:joinsemi}
        T_{j^*}=\mathcal{G}(T_{j^*-1},\widehat{r}_{j^*})=T_{j^*-1}\ldownarrow^{(\widehat{r}_{j^*},c_{j^*})}_{(r'_{j^*},c_{j^*})}\cup\{\langle \widehat{r}_{j^*},c_{j^*}\rangle\};
    \end{equation}
    note that, since $\widehat{D}\preceq T_{j^*}$, it follows that $\langle \widehat{r}_{j^*},c_{j^*}\rangle\in \widehat{D}$. Since $T_{j^*}\backslash T_{j^*-1}\neq T'_{j^*}\backslash T'_{j^*-1}$, one of the following holds
    \begin{enumerate}
        \item[$(1)$] $\widehat{r}_{j^*}=r_2$ and $c_2>c_{j^*}$, i.e., $\langle r_2,c_2\rangle$ lies to the right of $(\widehat{r}_{j^*},c_{j^*})$ in $T'_{j^*-1}$,
        \item[$(2)$] $\widehat{r}_{j^*}=r_2^*$ and $c_2>c_{j^*}$, i.e., $(r^*_2,c_2)$ lies to the right of $(\widehat{r}_{j^*},c_{j^*})$ in $T'_{j^*-1}$, or
        \item[$(3)$] $(r_2,c_2)$ or $(r_2^*,c_2)=(r'_{j^*},c_{j^*})$, i.e., $(r'_{j^*},c_{j^*})$ is not an empty position in $T'_{j^*-1}$.
    \end{enumerate}
    For (1), note that since $\langle r_2,c_2\rangle\in D_2$, it follows that $\langle r_2,c_2\rangle\in \tilde{D}$ for all $\tilde{D}\preceq D_2$. Moreover, since $\langle r_2,c_{j^*}\rangle=\langle \widehat{r}_{j^*},c_{j^*}\rangle\notin T_{j^*-1}\preceq D_0$ and $D_2\backslash D_0=\{\langle r_2,c_2\rangle,(r_2^*,c_2)\}$, we have that $\langle r_2,c_{j^*}\rangle\notin D_2$. Consequently, since ghost moves cannot introduce ghost cells to the right of an existing ghost cell, $\langle r_2,c_{j^*}\rangle\notin\tilde{D}$ for all $\tilde{D}\preceq D_2$; but $\widehat{D}\preceq D_2$ and, as noted above, $\langle r_2,c_{j^*}\rangle\in \widehat{D}$, which is a contradiction. Moving to (2), the argument is similar to (1) as, applying Lemma~\ref{lem:block} (a), $(r_2^*,c_2)\in D_2$ implies that $(r_2^*,c_2)$ or $\langle r_2^*,c_2\rangle\in \tilde{D}$ for all $\tilde{D}\preceq D_2$. Finally, for (3), note that neither case is possible:
    \begin{itemize}
        \item since $(r_2,c_2)\in T_{j^*-1}$ and position $(r_{j^*},c_{j^*})$ must be empty in $T_{j^*-1}$ considering \eqref{eq:joinsemi}, it follows that $(r'_{j^*},c_{j^*})\neq (r_2,c_2)$; and
        \item since $(r_2^*,c_2),\langle r_2^*,c_2\rangle\notin T_i$ for $i\in [j-1]$, as noted above, with $j^*-1<j-1$, it follows that $(r'_{j^*},c_{j^*})\neq (r_2^*,c_2)$.
    \end{itemize}
    Consequently, we have that $T_i\backslash T_{i-1}=T'_i\backslash T'_{i-1}$ for $i\in [j-1]$. Therefore, as noted above, $\widehat{D}\preceq T_j\preceq D_3$ and the result follows.
\end{proof}

Proposition~\ref{prop:meet} in hand, we are now in a position to establish the main result of this subsection.

\begin{theorem}\label{thm:jsl}
    For any diagram $D$ with $|G(D)|=0$, the poset $\mathcal{P}_G(D)$ is a join-semilattice.
\end{theorem}

\begin{proof}
Let $D$ be a diagram with $|G(D)|=0$. Suppose that $D_0,D_1,D_2\in \mathcal{P}_G(D)$ are distinct diagrams for which $D_1,D_2\precdot D_0$ and $$D_i = \mathcal{G}(D_0, r_i)=D_i\ldownarrow^{(r_i,c_i)}_{(r^*_i,c_i)}\cup\{\langle r_i,c_i\rangle\}$$ for $i=1,2$. Since $D_1\neq D_2$, it follows that $r_1\neq r_2$. Without loss of generality, assume that $r_1>r_2$. Considering Lemma~\ref{lem:nathan}, to establish the present result, it suffices to show that either 
\begin{enumerate}
    \item[$(a)$] $D_1$ and $D_2$ have a meet or
    \item[$(b)$] $D_1$ and $D_2$ have no common lower bound.
\end{enumerate}
There are five cases. For the diagrams illustrating the possible cases in Figure~\ref{fig:semilattice1}, regions shaded with diagonal lines represent regions in which no (initial) assumptions are made regarding the placement of cells.
\bigskip

\begin{figure}[H]
    \centering
    $$\scalebox{0.8}{\begin{tikzpicture}[scale=0.65]
    \node at (2.5, 2.5) {$\cdot$};
    \node at (2.5, 5.5) {$\cdot$};
  \draw (4.5,0.5)--(0.5,0.5)--(0.5,7.5);
  \draw (2,2)--(3,2)--(3,3)--(2,3)--(2,2);
  \draw (2,5)--(3,5)--(3,6)--(2,6)--(2,5);
    \path[pattern=north west lines] (0.5,0.5)--(4,0.5) -- (4,2) -- (2,2) -- (2,3) -- (4,3) -- (4,5)--(2,5)--(2,6)--(4,6)--(4,7)--(0.5,7)--(0.5,0.5)-- cycle;
  \node at (2.5, -0.25) {\large $c$};
  \node at (-0.25, 2.5) {\large $r_2$};
  \node at (-0.25, 5.5) {\large $r_1$};
  \node at (2, -1.5) {\large $(a)$};
\end{tikzpicture}}\quad\quad\quad\quad \scalebox{0.8}{\begin{tikzpicture}[scale=0.65]
    \node at (2.5, 2.5) {$\cdot$};
    \node at (5, 5.5) {$\cdot$};
  \draw (7,0.5)--(0.5,0.5)--(0.5,7.5);
  \draw (2,2)--(3,2)--(3,3)--(2,3)--(2,2);
  \draw (4.5,5)--(5.5,5)--(5.5,6)--(4.5,6)--(4.5,5);
    \path[pattern=north west lines] (0.5,0.5)--(6.5,0.5) -- (6.5,2) -- (2,2) -- (2,3) -- (6.5,3) -- (6.5,5)--(4.5,5)--(4.5,6)--(6.5,6)--(6.5,7)--(0.5,7)--(0.5,0.5)-- cycle;
  \node at (2.5, -0.25) {\large $c_2$};
  \node at (5, -0.25) {\large $c_1$};
  \node at (-0.25, 2.5) {\large $r_2$};
  \node at (-0.25, 5.5) {\large $r_1$};
  \node at (3.5, -1.5) {\large $(b)$};
\end{tikzpicture}}\quad\quad\quad\quad \scalebox{0.8}{\begin{tikzpicture}[scale=0.65]
    \node at (2.5, 5.5) {$\cdot$};
    \node at (5, 2.5) {$\cdot$};
  \draw (7,0.5)--(0.5,0.5)--(0.5,7.5);
  \draw (2,5)--(3,5)--(3,6)--(2,6)--(2,5);
  \draw (4.5,2)--(5.5,2)--(5.5,3)--(4.5,3)--(4.5,2);
    \path[pattern=north west lines] (0.5,0.5)--(6.5,0.5) -- (6.5,2) -- (4.5,2) -- (4.5,3) -- (6.5,3) -- (6.5,5)--(2,5)--(2,6)--(6.5,6)--(6.5,7)--(0.5,7)--(0.5,0.5)-- cycle;
  \node at (2.5, -0.25) {\large $c_1$};
  \node at (5, -0.25) {\large $c_2$};
  \node at (-0.25, 2.5) {\large $r_2$};
  \node at (-0.25, 5.5) {\large $r_1$};
  \node at (3.5, -1.5) {\large $(c)$};
\end{tikzpicture}}$$
    \caption{Theorem~\ref{thm:jsl} Cases}
    \label{fig:semilattice1}
\end{figure}
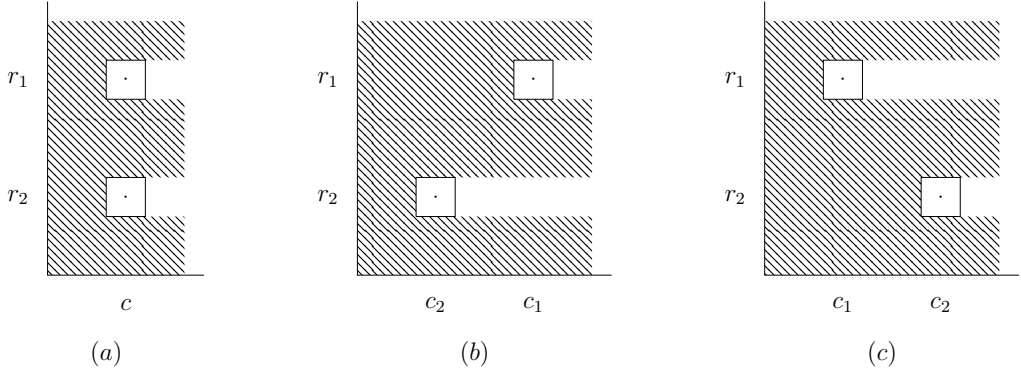
\noindent

\noindent
\textbf{Case 1:} $c_1=c_2=c$ and $(r,c)\in D_0$ for all $r_2\le r\le r_1$ (see Figure~\ref{fig:semilattice1} (a) where we are assuming that there are no empty positions between $(r_1,c)$ and $(r_2,c)$). In this case, we claim that $D_1$ and $D_2$ have no common lower bound. To see this, note that since $\langle r_1,c\rangle\in D_1$, it follows that $\langle r_1,c\rangle\in T$ for all $T\preceq D_1$. On the other hand, since $\langle r_2,c\rangle,(r,c)\in D_2$ for all $r_2<r\le r_1$, applying Lemma~\ref{lem:block} (b), we have that $(r_1,c)\in T$ for all $T\preceq D_2$. Consequently, $D_1$ and $D_2$ have no common lower bound, as claimed.
\bigskip

\noindent
\textbf{Case 2:} $c_1=c_2=c$ and there exists $r_2<r<r_1$ such that $(r,c)\notin D_0$ (see Figure~\ref{fig:semilattice1} (a) where we are assuming that there exists an empty position between $(r_1,c)$ and $(r_2,c)$); note that $r\le r^*_1$. In this case, we claim that $D_1$ and $D_2$ have a meet. To see this, note that, evidently, $$D_3=(D_0\backslash\{(r_1,c),(r_2,c)\})\cup\{(r^*_1,c),(r^*_2,c),\langle r_1,c\rangle,\langle r_2,c\rangle\}=\mathcal{G}(D_1,r_2)=\mathcal{G}(D_2,r_1)$$ and, considering Corollary~\ref{cor:cover}, $D_3\precdot D_1,D_2$. Thus, applying Proposition~\ref{prop:meet}, it follows that $D_3=D_1\wedge D_2$, i.e., $D_1$ and $D_2$ have a meet, as claimed.
\bigskip

\noindent
\textbf{Case 3:} $c_1>c_2$ and $r_1^*=r_2$ (see Figure~\ref{fig:semilattice1} (b) where we are assuming that there exists no empty positions between $(r_1,c_1)$ and $(r_2,c_1)$). In this case, we claim that $D_1$ and $D_2$ have no common lower bound. To see this, note that since $\langle r_2,c_2\rangle\in D_2$, it follows that $\langle r_2,c_2\rangle\in T$ for all $T\preceq D_2$. Now, as $(r_2,c_1)\in D_1$, applying Lemma~\ref{lem:block} (a), we have that $(r_2,c_1)$ or $\langle r_2,c_1\rangle\in T$ for all $T\preceq D_1$. Consequently, since $(r_2,c_2)\in D_1$ with $c_2<c_1$, it follows that $(r_2,c_2)\in T$ for all $T\preceq D_1$. Thus, $D_1$ and $D_2$ have no common lower bounds, as claimed.
\bigskip

\noindent
\textbf{Case 4:} $c_1>c_2$ and $r_1^*\neq r_2$ (see Figure~\ref{fig:semilattice1} (b) where we are assuming that there exists an empty position between $(r_1,c_1)$ and $(r_2,c_1)$). In this case, we claim that $D_1$ and $D_2$ have a meet. To see this, note that since $\mathcal{G}(D_0,r_2)=D_0\ldownarrow^{(r_2,c_2)}_{(r_2^*,c_2)}\cup\{\langle r_2,c_2\rangle\}$, it follows that $(r_2,\tilde{c})\notin D_0$ for $\tilde{c}>c$. Consequently, since $c_1>c_2$ and $r_1^*\neq r_2$, we have that $r_1^*>r_2$. Thus, $$D_3=(D_0\backslash\{(r_1,c),(r_2,c)\})\cup\{(r^*_1,c),(r^*_2,c),\langle r_1,c\rangle,\langle r_2,c\rangle\}=\mathcal{G}(D_1,r_2)=\mathcal{G}(D_2,r_1)$$ and, considering Corollary~\ref{cor:cover}, $D_3\precdot D_1,D_2$. Therefore, applying Proposition~\ref{prop:meet}, it follows that $D_3=D_1\wedge D_2$, i.e., $D_1$ and $D_2$ have a meet, as claimed.
\bigskip

\noindent
\textbf{Case 5:} $c_1<c_2$ (see Figure~\ref{fig:semilattice1} (c)). In this case, we claim that $D_1$ and $D_2$ have a meet. To see this, note that, evidently, $$D_3=(D_0\backslash\{(r_1,c),(r_2,c)\})\cup\{(r^*_1,c),(r^*_2,c),\langle r_1,c\rangle,\langle r_2,c\rangle\}=\mathcal{G}(D_1,r_2)=\mathcal{G}(D_2,r_1)$$ and, considering Corollary~\ref{cor:cover}, $D_3\precdot D_1,D_2$. Thus, applying Proposition~\ref{prop:meet}, it follows that $D_3=D_1\wedge D_2$, i.e., $D_1$ and $D_2$ have a meet, as claimed.
\bigskip

\noindent
The result follows.
\end{proof}

Now that we have established that $\mathcal{P}_G(D)$ is a join-semilattice for any diagram $D$ with $|G(D)|=0$, it is natural to ask when such posets are, in fact, lattices. Note that in the case that $\mathcal{P}_G(D)$ is a lattice, it must have a unique minimal element. Since $\mathcal{P}_G(D)$ always has a unique maximal element, i.e., $D$, it follows that if $\mathcal{P}_G(D)$ is a lattice, then it is bounded. Conversely, combining Theorem~\ref{thm:jsl} above with Proposition~\ref{prop:Stanley} below coming from~\cite{EC1}, it follows that if $\mathcal{P}_G(D)$ is bounded, then it is a lattice.

\begin{prop}[\cite{EC1}, Proposition 3.3.1]\label{prop:Stanley}
If $\mathcal{P}$ is a finite join-semilattice with a unique minimal element, then $\mathcal{P}$ is a lattice. 
\end{prop}

Thus, we have established the following.

\begin{theorem}
    Let $D$ be a diagram with $|G(D)|=0$. Then $\mathcal{P}_G(D)$ is a lattice if and only if it is bounded.
\end{theorem}

In the following section, we consider the question of characterizing diagrams $D$ with $|G(D)|=0$ for which $\mathcal{P}_G(D)$ is bounded, and, thus, a lattice.

\subsection{Bounded}

In this final subsection, given a diagram $D$ with $|G(D)|=0$, we establish a necessary condition for when $\mathcal{P}_G(D)$ is bounded. Our characterization is in terms of a sequence of integers associated with the free cells of a diagram, called the ``free cell sequence."

Given a diagram $D$ with $|G(D)|=0$, define $$FR(D)=\{r~|~\text{there exists}~\widehat{r},c\in\mathbb{Z}_{>0}~\text{such that}~\widehat{r}<r,~(r,c)\in D,~\text{and}~(\widehat{r},c),(r,\tilde{c})\notin D~\text{for all}~\tilde{c}>c\},$$ i.e., $FR(D)$ is the set of rows containing free cells in $D$. Letting $FR(D)=\{r_i\}_{i=1}^n$ with $r_i>r_{i+1}$ for $i\in [n-1]$ in the case $n>1$, define the \textbf{free cell sequence} of $D$ to be $(c_1,\hdots,c_n)$ where $c_i=\max\{c~|~(r_i,c)\in D\}$; note that $(r_i,c_i)$ is a free cell of $D$.



\begin{example}
    The diagrams $D_1$, $D_2$, and $D_3$ illustrated in Figure~\ref{fig:frs} below have free cell sequences $(1,2,2)$, $(1,2,1)$, and $(1,2,3)$, respectively.
    \begin{figure}[H]
        \centering
        $$\scalebox{0.8}{\begin{tikzpicture}[scale=0.6]
            \draw (0,5.5)--(0,0)--(2.5,0);
            \draw (0,5)--(1,5)--(1,4)--(0,4)--(0,5);
            \draw (0,4)--(1,4)--(1,3)--(0,3)--(0,4);
            \draw (1,4)--(2,4)--(2,3)--(1,3)--(1,4);
            \draw (1,2)--(2,2)--(2,1)--(1,1)--(1,2);
            \node at (0.5, 4.5) {$\cdot$};
            \node at (0.5, 3.5) {$\cdot$};
            \node at (1.5, 3.5) {$\cdot$};
            \node at (1.5, 1.5) {$\cdot$};
            \node at (1.5, -1){\large{$D_1$}};
        \end{tikzpicture}}\quad\quad\quad\quad\scalebox{0.8}{\begin{tikzpicture}[scale=0.6]
            \draw (0,4.5)--(0,0)--(2.5,0);
            \draw (0,4)--(1,4)--(1,3)--(0,3)--(0,4);
            \draw (1,3)--(2,3)--(2,2)--(1,2)--(1,3);
            \draw (0,2)--(1,2)--(1,1)--(0,1)--(0,2);
            \draw (1,1)--(2,1)--(2,0)--(1,0)--(1,1);
            \node at (0.5, 3.5) {$\cdot$};
            \node at (1.5, 2.5) {$\cdot$};
            \node at (0.5, 1.5) {$\cdot$};
            \node at (1.5, 0.5) {$\cdot$};
            \node at (1.5, -1){\large $D_2$};
        \end{tikzpicture}}\quad\quad\quad\quad\scalebox{0.8}{\begin{tikzpicture}[scale=0.6]
            \draw (0,4.5)--(0,0)--(4.5,0);
            \draw (0,4)--(1,4)--(1,3)--(0,3)--(0,4);
            \draw (1,3)--(2,3)--(2,2)--(1,2)--(1,3);
            \draw (2,2)--(3,2)--(3,1)--(2,1)--(2,2);
            \draw (3,1)--(4,1)--(4,0)--(3,0)--(3,1);
            \node at (0.5, 3.5) {$\cdot$};
            \node at (1.5, 2.5) {$\cdot$};
            \node at (2.5, 1.5) {$\cdot$};
            \node at (3.5, 0.5) {$\cdot$};
            \node at (2, -1){\large $D_3$};
        \end{tikzpicture}}$$
        \caption{Diagram Examples}
        \label{fig:frs}
    \end{figure}
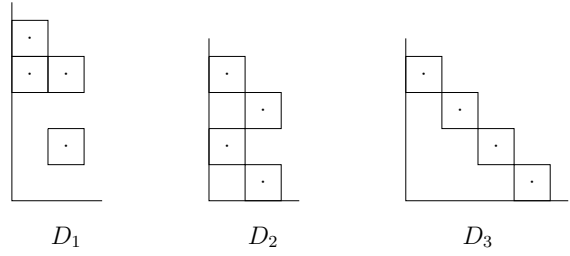
\end{example}

\begin{theorem}\label{thm:bounded-increase}
    Let $D$ be a diagram with $|G(D)|=0$. If $\mathcal{P}_G(D)$ is bounded, then the free cell sequence of $D$ is strictly increasing.
\end{theorem}
\begin{proof}
    We prove the contrapositive. Let $D$ be a diagram with $|G(D)|=0$ such that the free cell sequence $(c_1,\hdots,c_n)$ is not strictly increasing. Then there exists $i\in [n-1]$ such that $c_i\ge c_{i+1}$. By the definition of $(c_1,\hdots,c_n)$, there exists $r_i,r_{i+1}\in \mathbb{Z}_{>0}$ with $r_i>r_{i+1}$ such that $(r_i, c_i),(r_{i+1},c_{i+1})\in D$ are free cells in $D$. We claim that $(r,c)\notin D$ for all $c\ge c_{i+1}$ and $r_{i+1}<r<r_i$. To see this, assume otherwise. Then there exists $r^*,c^*\in\mathbb{Z}_{>0}$ such that $c^*\ge c_{i+1}$, $r_{i+1}<r^*<r_i$, and $(r^*,c^*)\in D$ is rightmost cell in row $r^*$ of $D$. Now, considering the definition of $(c_1,\hdots,c_n)$ once again along with the fact that $r_{i+1}<r^*<r_i$, it follows that $(r^*,c^*)$ is not a free cell of $D$. Thus, $(\tilde{r},c^*)\in D$ for all $1\le \tilde{r}\le r^*$; but, since $c^*\ge c_{i+1}$ and $r^*>r_{i+1}$, this implies that $(r_{i+1},c_{i+1})$ cannot be a free cell of $D$, which is a contradiction. Therefore, the claim follows. To finish the proof, we consider the cases of $c_i=c_{i+1}$ and $c_i>c_{i+1}$ separately.
    \bigskip

    \noindent
    \textbf{Case 1:} $c_i > c_{i+1}$. In this case, we claim that there exist minimal diagrams $D_1,D_2\in\mathcal{P}_G(D)$ such that $\langle r_{i+1},c_{i+1}\rangle\in D_1$ and $( r_{i+1},c_{i+1})\in D_2$, from which it follows that $\mathcal{P}_G(D)$ is not bounded. To establish the claim, let
    \begin{enumerate}
        \item $\widehat{D}_1$ be the diagram formed from $D$ by applying a single ghost move at row $r_{i+1}$, i.e., $\widehat{D}_1=\mathcal{G}(D, r_{i+1})$; and
        \item $\widehat{D}_2$ be the diagram formed from $D$ by sequentially applying a single ghost move at rows $r_i$ down to $r_{i+1}+1$ in decreasing order.
    \end{enumerate}
    Since $(r_{i+1},c_{i+1})$ is a free cell of $D$, it follows that $\langle r_{i+1},c_{i+1}\rangle\in \widehat{D}_1$ and, consequently, $\langle r_{i+1},c_{i+1}\rangle\in T$ for all $T\preceq\widehat{D}_1$; that is, there exists a minimal element $D_1\in \mathcal{P}_G(D)$ for which $D_1\preceq \widehat{D}_1$ and, as a result, $\langle r_{i+1},c_{i+1}\rangle\in D_1$. On the other hand, since $(r_{i},c_{i})$ is a free cell of $D$ and $(\tilde{r},c_i)\notin D$ for $r_{i+1}< \tilde{r}<r_i$, it follows that $(r_{i+1},c_{i+1}),(r_{i+1},c_i)\in \widehat{D}_2$. Thus, applying Lemma~\ref{lem:block} (a), it follows that $(r_{i+1},c_{i+1})\in T$ for all $T\preceq\widehat{D}_2$; that is, there exists a minimal element $D_2\in \mathcal{P}_G(D)$ for which $D_2\preceq \widehat{D}_2$ and, as a result, $(r_{i+1},c_{i+1})\in D_2$. Hence, the claim follows and, as noted above, $\mathcal{P}_G(D)$ is not bounded.
    \bigskip


    

    \noindent
    \textbf{Case 2}: $c_i = c_{i+1}=c$. In this case, we claim that there exist minimal diagrams $D_1,D_2\in\mathcal{P}_G(D)$ such that $(r_{i+1}-1,c)\in D_1$ and $\langle r_{i+1}-1,c\rangle\in D_2$, from which it follows that $\mathcal{P}_G(D)$ is not bounded. To establish the claim, let
    \begin{enumerate}
        \item $\widehat{D}_1$ be the diagram formed from $D$ by applying a single ghost move at row $r_{i+1}$ and then sequentially applying a single ghost move at rows $r_i$ down to $r_{i+1}+1$ in decreasing order; and 
        \item $\widehat{D}_2$ be the diagram formed from $D$ by sequentially applying a single ghost move at rows $r_i$ down to $r_{i+1}+1$ in decreasing order.
    \end{enumerate}
    Since $(r_{i},c),(r_{i+1},c)$ are free cells of $D$ and $(\tilde{r},c)\notin D$ for $r_{i+1}<\tilde{r}<r_i$, it follows that $\langle r_{i+1},c\rangle, (r_{i+1}+1,c)\in \widehat{D}_1$ and $\langle r_{i+1}+1,c\rangle\in \widehat{D}_2$.
    Thus, it follows that $(r_{i+1}+1,c)\in T$ (resp., $\langle r_{i+1}+1,c\rangle\in T$) for all $T\preceq\widehat{D}_1$ (resp., $T\preceq\widehat{D}_2$); that is, there exists a minimal element $D_1\in \mathcal{P}_G(D)$ (resp., $D_2\in \mathcal{P}_G(D)$) for which $D_1\preceq \widehat{D}_1$ (resp., $D_2\preceq \widehat{D}_2$) and, as a result, $(r_{i+1}+1,c)\in D_1$ (resp., $\langle r_{i+1}+1,c\rangle\in D_2$). Hence, the claim follows and, as noted above, it follows that $\mathcal{P}_G(D)$ is not bounded. \qedhere
    
    
    
\end{proof}

Note that the free cell sequence of a diagram $D$ is strictly increasing if and only if there is at most one free cell per column in $D$ and the free cells in $D$ occur in descending row order from left to right. Therefore, we can equivalently express Theorem~\ref{thm:bounded-increase} as follows.

\begin{corollary}\label{cor:b2}
    Let $D$ be a diagram with $|G(D)|=0$. If $\mathcal{P}_G(D)$ is bounded, then $D$ has at most one free cell per column and the free cells occur in descending row order from right to left.
\end{corollary}

Unfortunately, a diagram $D$ satisfying the condition of Theorem~\ref{thm:bounded-increase} (equivalently, Corollary~\ref{cor:b2}) is not sufficient for $\mathcal{P}_G(D)$ to be bounded, as shown in the following example.

\begin{example}\label{not-bounded}    
The free cell sequence of the diagram $D$ shown in Figure~\ref{fig:bounded} $(a)$ is $(2,3)$, which is strictly increasing, yet $\mathcal{P}(D)$ contains the two distinct minimal elements illustrated in Figure~\ref{fig:bounded} $(b)$ and $(c)$. Thus, $D$ satisfies the conditions of Theorem~\ref{thm:bounded-increase} \textup(equivalently, Corollary~\ref{cor:b2}\textup), but $\mathcal{P}_G(D)$ is not bounded.

     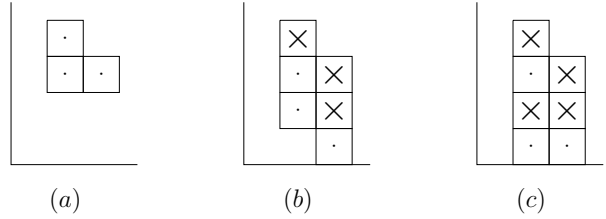
\begin{figure}[H]
         \centering
         $$\scalebox{0.8}{\begin{tikzpicture}[scale=0.6]
            \draw (0,4.5)--(0,0)--(3.5,0);
            \draw (1,4)--(2,4)--(2,3)--(1,3)--(1,4);
            \draw (1,3)--(2,3)--(2,2)--(1,2)--(1,3);
            \draw (2,3)--(3,3)--(3,2)--(2,2)--(2,3);
            \node at (1.5, 3.5) {$\cdot$};
            \node at (1.5, 2.5) {$\cdot$};
            \node at (2.5, 2.5) {$\cdot$};
            \node at (1.5, -1) {\large $(a)$};
            \end{tikzpicture}}\quad\quad\quad\quad\scalebox{0.8}{\begin{tikzpicture}[scale=0.6]
            \draw (0,4.5)--(0,0)--(3.5,0);
            \draw (1,4)--(2,4)--(2,3)--(1,3)--(1,4);
            \draw (1,3)--(2,3)--(2,2)--(1,2)--(1,3);
            \draw (1,2)--(2,2)--(2,1)--(1,1)--(1,2);
            \draw (2,3)--(3,3)--(3,2)--(2,2)--(2,3);
            \draw (2,2)--(3,2)--(3,1)--(2,1)--(2,2);
            \draw (2,1)--(3,1)--(3,0)--(2,0)--(2,1);
            \node at (1.5, 3.5) {$\bigtimes$};
            \node at (1.5, 2.5) {$\cdot$};
            \node at (1.5, 1.5) {$\cdot$};
            \node at (2.5, 2.5) {$\bigtimes$};
            \node at (2.5, 1.5) {$\bigtimes$};
            \node at (2.5, 0.5) {$\cdot$};
            \node at (1.5, -1) {\large $(b)$};
        \end{tikzpicture}} \quad\quad\quad\quad \scalebox{0.8}{\begin{tikzpicture}[scale=0.6]
            \draw (0,4.5)--(0,0)--(3.5,0);
            \draw (1,4)--(2,4)--(2,3)--(1,3)--(1,4);
            \draw (1,3)--(2,3)--(2,2)--(1,2)--(1,3);
            \draw (1,2)--(2,2)--(2,1)--(1,1)--(1,2);
            \draw (1,1)--(2,1)--(2,0)--(1,0)--(1,1);
            \draw (2,3)--(3,3)--(3,2)--(2,2)--(2,3);
            \draw (2,2)--(3,2)--(3,1)--(2,1)--(2,2);
            \draw (2,1)--(3,1)--(3,0)--(2,0)--(2,1);
            \node at (1.5, 3.5) {$\bigtimes$};
            \node at (1.5, 2.5) {$\cdot$};
            \node at (1.5, 1.5) {$\bigtimes$};
            \node at (1.5, 0.5) {$\cdot$};
            \node at (2.5, 2.5) {$\bigtimes$};
            \node at (2.5, 1.5) {$\bigtimes$};
            \node at (2.5, 0.5) {$\cdot$};
            \node at (1.5, -1) {\large $(c)$};
        \end{tikzpicture}}$$
         \caption{Two unique minimal elements of $\mathcal{P}(D)$}
         \label{fig:bounded}
     \end{figure}
\end{example}

\begin{remark}
    It is worth noting that the diagram $D$ of Figure~\ref{fig:bounded} $(a)$ satisfies the sufficient condition of \textup{\cite{KPoset1}} for the Kohnert poset of $D$, i.e., $\mathcal{P}(D)$, to be bounded. Thus, $\mathcal{P}(D)$ is bounded, while $\mathcal{P}_G(D)$ is not bounded.
\end{remark}

For future work, it would be interesting if one could strengthen the necessary condition of Theorem~\ref{thm:bounded-increase} (equivalently, Corollary~\ref{cor:b2}) to provide a characterization for when $\mathcal{P}_G(D)$ is bounded completely in terms of $D$. The authors feel that making use of a labeling of the diagram $D$ could prove helpful. For an example of a labeling, see Section 4 of \cite{KPuzzle} where a labeling is used to compute the max number of ghost cells contained in a diagram of $\mathcal{P}_G(D)$; at present, the authors are unaware as to whether this particular labeling can be applied to the problem of establishing when $\mathcal{P}_G(D)$ is bounded. As noted in the previous section, a characterization of boundedness would also provide a characterization for when $\mathcal{P}_G(D)$ is a lattice.

In addition, one could also explore the polynomials that can be naturally associated with the posets $\mathcal{P}_G(D)$ via a construction similar to that given for Lascoux polynomials given in Remark~\ref{rem:Lascoux}, i.e., $$\mathfrak{G}_D=\sum_{T\in \mathrm{GKD}(D)}\mathrm{wt}^+(T)$$ where $\mathrm{wt}^+(T)=\prod_{r\ge 1}(-1)^{|G(T)|}x_r^{|\{c~|~(r,c)~\text{or}~\langle r,c\rangle\in T\}|}.$ Other than investigating whether or not such polynomials have any desirable properties or interpretations, one could also explore the question of whether there are any relationships between poset properties of $\mathcal{P}_G(D)$ and polynomial properties of $\mathfrak{G}_D$.

\printbibliography

\end{document}